\documentclass[a4paper,12pt]{article}

\usepackage{color,makeidx,bm,latexsym,amscd,amsmath,amssymb,enumerate,stmaryrd,amsthm,float,graphicx,psfrag,geometry,mathrsfs,mathtools}

\usepackage[all]{xy}

\usepackage{tikz,epic,eso-pic}

\usetikzlibrary { decorations.pathmorphing, decorations.pathreplacing, decorations.shapes, }

\theoremstyle{plain}
\newtheorem{theorem}{Theorem}[section]
\newtheorem{corollary}[theorem]{Corollary}
\newtheorem{lemma}[theorem]{Lemma}

\theoremstyle{definition}
\newtheorem{definition}[theorem]{Definition}
\newtheorem{remark}[theorem]{Remark}
\newtheorem{example}[theorem]{Example}
\newtheorem{proposition}[theorem]{Proposition}

\newcommand{\C}{\mathbb{C}}

\newcommand{\Q}{\mathbb{Q}}
\newcommand{\R}{\mathbb{R}}
\newcommand{\Z}{\mathbb{Z}}

\newcommand{\calF}{\mathcal{F}}
\newcommand{\calK}{\mathcal{K}}
\newcommand{\calM}{\mathcal{M}}

\newcommand{\calP}{\mathcal{P}}

\newcommand{\Tor}{\operatorname{Tor}} 

\newcommand{\Hahn}{\mathscr{H}} 
\newcommand{\Mag}{\operatorname{Mag}} 
\newcommand{\MH}{\operatorname{MH}} 
\newcommand{\MC}{\operatorname{MC}} 
\newcommand{\Cau}{\operatorname{Cau}} 
\newcommand{\ess}{\operatorname{ess}} 
\newcommand{\Deltadot}{\Delta'} 
\newcommand{\inter}{\operatorname{int}} 
\newcommand{\emptycpx}{\{\emptyset\}} 
\newcommand{\void}{{\bf void}} 

\newcommand{\Top}{\operatorname{\sf Top}} 
\newcommand{\Met}{\operatorname{\sf Met}} 
\newcommand{\rank}{\operatorname{rank}}

\newcommand{\Sus}{\operatorname{\sf \Sigma}} 
\newcommand{\Cone}{\operatorname{\Gamma}} 

\title{Causal order complex and magnitude homotopy type of metric spaces} 
\author{Yu Tajima\thanks{Department of Mathematics, Graduate School of Science, 
Hokkaido University, North 10, West 8, Kita-ku, 
Sapporo 060-0810, JAPAN 
E-mail: tajima@math.sci.hokudai.ac.jp}, 
Masahiko Yoshinaga\thanks{Department of Mathematics, Graduate School of Science, Osaka University, Toyonaka 560-0043, JAPAN 
E-mail: yoshinaga@math.sci.osaka-u.ac.jp}}
\date{\today}
\begin{document}
\maketitle

\begin{abstract} 
In this paper, we construct a pointed CW complex 
called the magnitude homotopy type for a given metric space $X$ 
and a real parameter $\ell \geq 0$. 
This space is roughly consisting of all paths of length $\ell$ and 
has the reduced homology group that is isomorphic to the 
magnitude homology group of $X$. 

To construct the magnitude homotopy type, we consider 
the poset structure on the spacetime 
$X\times\mathbb{R}$ defined by causal (time- or light-like) relations. 
The magnitude homotopy type 
is defined as the quotient of the order complex of an intervals on 
$X\times\mathbb{R}$ by a certain subcomplex. 

The magnitude homotopy type gives a covariant functor from 
the category of metric spaces with $1$-Lipschitz maps 
to the category of pointed topological spaces. 
The magnitude homotopy type also 
has a ``path integral'' like expression for certain metric spaces. 

By applying discrete Morse theory to the magnitude homotopy type, 
we obtain a new proof of the Mayer-Vietoris type theorem and several new 
results including the invariance of the magnitude under sycamore twist of 
finite metric spaces. 
\end{abstract}

\tableofcontents

\section{Introduction}

The concept of magnitude, introduced by Leinster \cite{Lei-met}, is a real valued 
cardinality-like numerical invariant for metric spaces. 
The magnitude is not just ``the rough size'' of metric spaces but, 
considered to be ``the effective number'' of points in the space, 
since it has fine structures e.g., additivity, multiplicativity, invariance under 
the Whitney twist (\cite{L}). 
The study of magnitude from analytic perspectives is also very active 
\cite{bar-car, gim-gof, mec-mag, wil-odd}. See \cite{lei-mec} for overviews of 
recent developments in magnitudes. 

Later, Hepworth-Willerton \cite{HW} 
and Leinster-Shulman \cite{lei-shu} have proposed the notion of the 
magnitude homology group as a categorification of the magnitude. 
It has been established that for closed sets in Euclidean space, 
the magnitude homology reflects properties such as convexity \cite{lei-shu} 
and the diameter of a hole \cite{KY}. 
In recent years, the notion of magnitude homology group 
has been studied in relation to various research topics such as 
path homology of graphs \cite{a-path}, random graphs \cite{ahk}, 
topological invariants of point clouds \cite{gov-hep, alpha}, 
magnitude cohomology \cite{hep-coh} etc. 
However, the information captured by the magnitude homology 
for general metric spaces, including finite ones, is not well understood.

The main objective of this paper is to investigate the causal poset structure of 
metric spaces and use it to construct the topological space which is 
referred to as 
\emph{the magnitude homotopy type}. 
We also apply the magnitude homotopy type to several problems. 

It should be noted that this space has already been dealt with 
in previous works by Hepworth-Willerton \cite{HW} (for graphs) and 
Bottinelli-Kaiser \cite{bot-kai} (for metric spaces). 
Asao-Izumihara \cite{AI} has also constructed a closely related space for 
graphs. 
One of the most crucial points in \cite{AI} was the introduction of the 
``time parameter $t$'' to specify the vertex $(x_t, t)$ in the simplicial complex.

In this paper, we provide an interpretation of Asao-Izumihara type 
space using the ``\emph{Causal poset structures}'' on 
the metric spacetime. 
The magnitude homotopy type can be constructed as a quotient of 
the order complex of causal poset by a certain subcomplex. 
Our construction gives an explicit presentation of the magnitude homotopy type as 
a pair of simplicial complexes. 
Thanks to the new presentation, it has become possible to apply 
tools from poset topology, particularly discrete Morse theory, 
to study the magnitude and magnitude homology.

The basic idea of the magnitude homotopy type $\calM^\ell(X; a, b)$ 
is to consider the space of all paths in $X$ that start 
at point $a$ at time $t=0$ and end at point $b$ at time $t=\ell$, 
where $X$ is a metric space and $a, b\in X$ ($a=b$ is allowed). 
The precise formulation of $\calM^\ell(X; a, b)$ 
is based on the order complex of the poset 
defined by causal order structures on $X\times\R$ 
(See \S \ref{sec:pre}, \S \ref{sec:causal}, and \S \ref{sec:mht} 
for details). 

Let us now summarize the formal aspects of the magnitude 
homotopy type. 
We denote the category of $2$-pointed metric spaces with 
$1$-Lipschitz maps by $\Met_{**}^1$. Namely, an object of 
$\Met_{**}^1$ is a tuple $(X; a, b)$, 
and a morphism $(X; a, b)\stackrel{f}{\longrightarrow}(Y, c, d)$ 
is a continuous map $f:X\longrightarrow Y$ which satisfies 
$d_Y(f(x), f(x'))\leq d_X(x, x')$, $f(a)=c$ and $f(b)=d$. 
Let $\ell\in\R_{\geq 0}$. 
The magnitude homotopy type $\calM^\ell$ is a covariant functor 
\begin{equation}
\calM^\ell:\Met_{**}^1\longrightarrow\Top_*, \ \ 
(X; a, b)\longmapsto \calM^\ell(X; a, b), 
\end{equation}
from $\Met_{**}^1$ to the category of pointed topological spaces 
$\Top_*$. 

Besides functoriality, the magnitude homotopy type also possesses 
several nice properties as follows. 
\begin{itemize}
\item 
The magnitude homology group $\MH^\ell_*(X)$ is isomorphic 
to the direct sum of reduced homology groups 
$\bigoplus_{a, b\in X}\widetilde{H}_*(\calM^\ell(X; a, b))$ 
(Theorem \ref{thm:mainisom}). 
\item 
The Euler characteristics 
$\{\widehat{\chi}(\calM^\ell(X; a, b))\}_{\ell\in\R_{\geq 0}, 
a, b\in X}$ recover the metric space $X$ 
(Theorem \ref{thm:recover}). 
\item 
The magnitude homotopy type is the double suspension of 
the space constructed by Asao-Izumihara \cite{AI} 
(Theorem \ref{thm:double}). 
\item 
The magnitude homotopy type satisfies K\"unneth formula 
as established in \cite{HW, bot-kai}. 
(Theorem \ref{thm:kunneth}). 
\item 
The magnitude homotopy type $\calM^\ell(X; a, b)$ behaves 
similarly to the probability amplitude in quantum mechanics and 
has a ``path integral'' like expression 
(Theorem \ref{thm:dec}, Remark \ref{rem:pertexp} and 
comments around (\ref{eq:inverse2})). 
\item 
The magnitude homotopy type satisfies a type of excision and 
Mayer-Vietoris formula 
as established in \cite{HW, bot-kai} (Proposition \ref{prop:decompH}, 
Theorem \ref{thm:union}, Corollary \ref{cor:mv}). 
\end{itemize}

The paper is organized as follows: 
In \S \ref{sec:MH}, we review the definitions of the magnitude for 
a finite metric space and 
the magnitude homology group for a general metric space. 
In \S \ref{sec:DMT}, we recall basic definitions and results 
on discrete Morse theory for simplicial complexes. 
As long as the authors know, many previous works on 
discrete Morse theory assume finiteness of simplicial 
complexes (Theorem \ref{thm:morse}). 
However, in this paper, we need an infinite simplicial complex 
version of the ``Main theorem of Discrete Morse Theory''. 
Therefore, we formulate and prove the version we need. 

In \S \ref{sec:causal}, we define the \emph{causal poset} 
structure on $X\times\R$, 
motivated by the \emph{time-like-} or \emph{light-like-} 
relation in the Minkowski spacetime. 
In \S \ref{sec:const}, we define the \emph{magnitude homotopy type} 
$\calM^\ell(X)$ of a metric space $X$ 
as a pair (or quotient) of the order complex of an interval of the causal poset and 
a certain subcomplex. 
The remainder of \S \ref{sec:mht} and \S \ref{sec:applic}, with the exception 
of \S \ref{sec:inv}, is devoted to proving the properties listed above. 
In \S \ref{sec:inv}, we apply discrete Morse theory to the magnitude homotopy type, 
and prove that the magnitude is invariant under sycamore twist, which generalizes 
a recent result by Roff \cite{roff}. 

\section{Preliminaries}
\label{sec:pre}

\subsection{Magnitude and magnitude homology}
\label{sec:MH}

In this section, we review the notion of the magnitude and the magnitude homology groups for 
metric spaces. 
Let $X$ be a finite metric space. We define the square matrix 
$Z_X$ to be 
\begin{equation}
Z_X=\left(q^{d(x, y)}\right)_{x, y\in X}, 
\end{equation}
where $q=e^{-1}$ (however, we consider $q$ to be 
a formal variable below). 
Since $d(x, x)=0$, the diagonal entries of $Z_X$ are all $1$, 
and the off-diagonal entries 
have strictly positive powers. The entries of the formal sum 
$\sum_{k=0}^\infty (I-Z_X)^k$ 
converge in the ring of Hahn series $\Hahn$ which is defined as 
\[
\Hahn=\left\{\left.\sum_{r\in\R_{\geq 0}}c_rq^r\right| c_r\in\C, 
\mbox{ the support }
\{r\in\R_{\geq 0}\mid c_r\neq 0\}\mbox{ is well-ordered}\right\}. 
\]
The ring $\Hahn$ is endowed with 
the $q$-adic topology, namely, the topology with 
a basis $\{f+q^r\Hahn\mid f\in\Hahn, r>0\}$. 
Let $r_0=\min\{d(x, y)\mid x, y\in X, x\neq y\}$. 
Then the entries 
of a matrix $I-Z_X$ in contained in $q^{r_0}\Hahn$.  
Since $r_0>0$, the right-hand side of 
\begin{equation}
\label{eq:inverse1}
Z_X^{-1}=(I-(I-Z_X))^{-1}=\sum_{k=0}^\infty (I-Z_X)^k
\end{equation}
converges in $\Hahn$. 
We denote by $Z_X^{-1}(x, y)\in\Hahn$ the $(x, y)$ entry 
of the matrix $Z_X^{-1}$. 
Note that (\ref{eq:inverse1}) is equivalent to the following 
formula\footnote{This formula can be thought of as analogous to the 
perturbative expansion of the amplitude of a particle 
interacting with a potential 
\cite[\S 6.2, (6.17)]{fey-hib}.}  
\begin{equation}
\label{eq:inverse2}
Z_X^{-1}(x, y)=
\sum_{k=0}^\infty(-1)^k\cdot
\sum_{\substack{a_0, \dots, a_k\in X\\ a_0=x, a_k=y, \\a_{i-1}\neq a_i (i=1, \dots, k)}}
e^{d(a_0, a_1)+\dots+d(a_{k-1}, a_k)}. 
\end{equation}
Then the \emph{magnitude weighting} $w:X\longrightarrow\Hahn$ and the \emph{magnitude} 
$\Mag(X)\in\Hahn$ is defined as 
\begin{equation}
\begin{split}
w(x)&:=\sum_{y\in X}Z_X^{-1}(x, y), \\
\Mag(X)&:=\sum_{x\in X}w(x)=\sum_{x, y\in X}Z_X^{-1}(x, y). 
\end{split}
\end{equation}

Next we define the magnitude homology group. Let $X$ be a metric space (not necessarily finite). 
We say that $\bm{x}=(x_0, x_1, \cdots, x_k)\in X^{k+1}$ 
is a \emph{sequence of degree $k$} if $x_{i-1}\neq x_{i}$ for any $i\in\{1, 2, \cdots, k\}$. 
Let $\bm{x}=(x_0, x_1, \cdots, x_k)\in X^{k+1}$ be a sequence. 
The length of a sequence 
$\bm{x}=(x_0, \dots, x_k)$ is defined as $d(\bm{x}):=d(x_0, x_1)+d(x_1, x_2)+\cdots+d(x_{k-1}, x_k)$.  

\begin{definition}[Magnitude homology of 
metric spaces \cite{HW, lei-shu}]
\label{def:MC}
Fix $\ell\geq 0$.
Define the abelian group $\MC_k^\ell(X)$ and the map $\partial$ as follows.
\begin{equation}
\label{eq:bdry}
\begin{split}
&\MC_k^\ell(X):=\bigoplus_{\substack{\bm{x}=(x_0, \cdots, x_k)\in X^{k+1}\\
\text{ sequence with }d(\bm{x})=\ell}} 
\Z\bm{x}, \\
&\partial\colon \MC_k^\ell(X)\rightarrow \MC_{k-1}^\ell(X),\ \partial:=\sum_{i=1}^{k-1}(-1)^i\partial_i, \\
&\partial_i(x_0, \cdots, x_k):=
\begin{cases}
(x_0, \cdots, \widehat{x_i}, \cdots, x_k),&
\text{ if } d(x_0, \cdots, \widehat{x_i}, \cdots, x_k)=\ell, \\
0, &\text{ otherwise}.
\end{cases}\\
\end{split}
\end{equation}
Then $(\MC_*^\ell(X), \partial)$ is a chain complex and it is called the magnitude chain complex. The \emph{magnitude homology} of $X$ is defined as the homology of the chain 
complex: $\MH_k^\ell(X):=H_k(\MC_*^\ell(X))$. 
\end{definition}
Let $a, b\in X$. Then we can define magnitude chain complex and magnitude homology 
group using only the sequences from $a$ to $b$, 
\begin{equation}
\MC_k^\ell(X; a, b)=\bigoplus_{\substack{\bm{x}=(x_0, \dots, x_k),\\ x_0=a, x_k=b}}\Z\bm{x}, 
\MH_k^\ell(X; a, b)=H_k(\MC_*^\ell(X; a, b)). 
\end{equation}
Furthermore, we have the following direct sum decomposition. 
\begin{equation}
\MH_k^\ell(X)=\bigoplus_{a, b\in X}\MH_k^\ell(X; a, b). 
\end{equation}

Now again we return to the case where $X$ is a finite metric space, and describe 
the relationship between the magnitude and the magnitude homology. Let $\ell\geq 0$. 
Then there are only finitely many sequences $\bm{x}=(x_0, \dots, x_k)$ with $d(\bm{x})=\ell$. 
Hence the Euler characteristic of magnitude homology can be defined and satisfies 
\begin{equation}
\sum_{k\geq 0}(-1)^k\rank\MH_k^\ell(X; a, b)=\sum_{k\geq 0}(-1)^k\rank\MC_k^\ell(X; a, b). 
\end{equation}
Using this formula, we have the following. 
\begin{equation}
\begin{split}
Z_X^{-1}(a, b)&=\sum_{\ell\geq 0}\left(\sum_{k\geq 0}(-1)^k\rank\MH_k^\ell(X; a, b)\right)q^\ell, \\
w(a)&=\sum_{\ell\geq 0, b\in X}
\left(\sum_{k\geq 0}(-1)^k\rank\MH_k^\ell(X; a, b)\right)q^\ell, \\
\Mag(X)&=\sum_{\ell\geq 0}\left(\sum_{k\geq 0}(-1)^k\rank\MH_k^\ell(X)\right)q^\ell. 
\end{split}
\end{equation}

\subsection{Discrete Morse theory and order complexes of posets}
\label{sec:DMT}

In this subsection we recall discrete Morse theory on simplicial complexes 
(see \cite{K} for details). 
Let $V$ be a nonempty set and $S\subset 2^V$ be 
a simplicial complex, i.e. , a collection of 
nonempty subsets of $V$ such that for any 
$a, b\subset V$,  $a\subset b$ and $b\in S$ imply $a\in S$. 
The geometric realization of $S$ is denoted by $|S|$. 

\begin{definition}[partial matching {\cite[Definition 11.1]{K}}]
A partial matching $M$ is a subset $M\subseteq S\times S$ 
satisfying the followings. 
\begin{itemize}
\item 
If $(b, a)\in M$, then $b\subset a$ and $|a\smallsetminus b|=1$. 
\item 
Each $a\in S$ belongs to at most one element in $M$. 
\end{itemize}
If $(b, a)\in M$, let us denote $b=d(a)$, $a=u(b)$, or $b\vdash a$. 
\end{definition}

\begin{definition}[acyclic matching]
A partial matching $M$ on $S$ is said to be \emph{acyclic} 
if there does not exist a cycle 
\begin{equation}
\label{eq:cycle}
a_1\supset b_1\vdash a_2\supset b_2\vdash\cdots\vdash a_p\supset b_p\vdash a_{p+1}=a_1
\end{equation}
with $p\geq 2$, and $a_i\neq a_j$ for every $i, j\in\{1, 2, \cdots, p\}$ with $i\neq j$.
\end{definition}

\begin{definition}[critical simplex]
Let $M$ be an acyclic matching on $S$. 
A simplex $a\in S$ is called a \emph{critical simplex} if 
$a$ does not belong to any element in $M$.
\end{definition}

\begin{theorem}[\cite{K}, Theorem 11.13 ``Main theorem of Discrete Morse Theory'']
\label{thm:morse}
Let $M$ be an acyclic matching on a finite simplicial complex $S$. 
\begin{itemize}
\item[(a)] 
If the critical cells form a subcomplex $S_c$, 
then there exists a sequence of cellular collapses leading 
from $|S|$ to $|S_c|$, in particular, $|S_c|$ is a deformation retract of $S$. 
\item[(b)] 
Denote the number of critical $i$-dimensional simplices by $c_i$.  
Then, $S$ is homotopy equivalent to a CW complex with $c_i$ cells in dimension $i$. 
\end{itemize}
\end{theorem}
In much of the literature on discrete Morse theory, it is assumed that 
the simplicial complex is finite. 
However, in this paper, we also require a version of 
Theorem \ref{thm:morse} (a) for infinite simplicial complexes. 
In this case, certain finiteness conditions on acyclic matching 
need to be imposed. 

\begin{definition}
\label{def:bdd}
Let $M$ be an acyclic matching on a simplicial complex $S$ (possibly 
infinite simplicial complex). We say $M$ is \emph{bounded} if 
for each simplex $a\in S$, there 
exists positive integer $N(a)>0$ such that for every sequence 
\begin{equation}
\label{eq:collapse}
a=a_1\supset b_1\vdash a_2\supset b_2\vdash\cdots\vdash a_p\supset b_p
\end{equation}
with $a_i\neq a_j, b_i\neq b_j$ for $i\neq j$, 
the length satisfies $p\leq N(a)$. 
\end{definition}

\begin{example}
\label{ex:infinitematching}
Consider the simplicial decomposition of $\R_{\geq 0}$ defined by 
$0$-simplices $\sigma_n=\{n\}$ ($n\in\Z_{\geq 0}$), and 
$1$-simplices $\tau_n=\{n, n+1\}$ ($n\in\Z_{\geq 0}$) as in 
Figure \ref{fig:unbdd}. Define 
acyclic matchings $M_1$ and $M_2$ by 
\[
\begin{split}
M_1&=\{(\sigma_i\vdash\tau_{i-1})\mid i\in\Z_{>0}\}, \\
M_2&=\{(\sigma_i\vdash\tau_{i})\mid i\in\Z_{\geq 0}\}. 
\end{split}
\]
\begin{figure}[htbp]
\centering
\begin{tikzpicture}


\draw[thick] (-0.5,1.5) node[left]{$M_1: $};
\filldraw[fill=black, draw=black] (0, 1.5) circle (2pt) node[above]{$\sigma_0$};
\filldraw[fill=black, draw=black] (2, 1.5) circle (2pt) node[above]{$\sigma_1$};
\filldraw[fill=black, draw=black] (4, 1.5) circle (2pt) node[above]{$\sigma_2$};
\filldraw[fill=black, draw=black] (6, 1.5) circle (2pt) node[above]{$\sigma_3$};
\draw (1,1.5) node[above]{$\subset \tau_0 \dashv$};
\draw (3,1.5) node[above]{$\subset \tau_1 \dashv$};
\draw (5,1.5) node[above]{$\subset \tau_2 \dashv$};
\draw[thick] (0,1.5)--(6.5,1.5) node[right]{$\cdots$}; 

\draw[thick] (-0.5,0) node[left]{$M_2: $};
\filldraw[fill=black, draw=black] (0, 0) circle (2pt) node[above]{$\sigma_0$};
\filldraw[fill=black, draw=black] (2, 0) circle (2pt) node[above]{$\sigma_1$};
\filldraw[fill=black, draw=black] (4, 0) circle (2pt) node[above]{$\sigma_2$};
\filldraw[fill=black, draw=black] (6, 0) circle (2pt) node[above]{$\sigma_3$};
\draw (1,0) node[above]{$\vdash \tau_0 \supset$};
\draw (3,0) node[above]{$\vdash \tau_1 \supset$};
\draw (5,0) node[above]{$\vdash \tau_2 \supset$};
\draw[thick] (0,0)--(6.5,0) node[right]{$\cdots$}; 

\end{tikzpicture}
\caption{Bouded and unbounded acyclic matchings.}
\label{fig:unbdd}
\end{figure}
Then $M_1$ is a bounded acyclic matching which has the unique critical 
simplex $\sigma_0$. $M_2$ is not a bounded acyclic matching 
which has no critical simplices. 
\end{example}

\begin{lemma}
\label{lem:finite}
Let $M$ be a bounded acyclic matching on a simplicial complex $S$. 
Then for every finite subcomplex $F\subset S$, there exists finite simplicial 
complex $\widetilde{F}$ such that 
\begin{itemize}
\item $F\subset\widetilde{F}\subset S$, 
\item if $b\in\widetilde{F}$ and $b\vdash a'\in S$, then $a'\in\widetilde{F}$. 
\end{itemize}
\end{lemma}
\begin{proof}
By the boundedness of $M$, there exist only finitely many sequences 
$a_1\supset b_1\vdash a_2\supset b_2\vdash\cdots\vdash a_p\supset b_p$ 
with $a_1\in F$ and $a_i, b_i\in S$. Let $\widetilde{F}$ denote the 
collection of all such $a_i, b_i\in S$, which satisfies the required conditions. 
\end{proof}

\begin{proposition}
\label{prop:infinite}
Let $M$ be a bounded acyclic matching on a connected simplicial complex $S$. 
If the critical cells form a subcomplex $S_c$, then $|S_c|$ is a deformation retract of $|S|$. 
\end{proposition}
\begin{proof}
Let $X=|S_c|$, $Y=|S|$, and $i: X\longrightarrow Y$ be the inclusion. 
Let $x_0\in X$. 
By Whitehead's theorem \cite[Theorem 4.5]{hat}, 
it is sufficient to show that 
$X$ is connected and the induced maps between homotopy groups 
$i_*:\pi_n(X, x_0)\longrightarrow\pi_n(Y, x_0)$ are 
isomorphic for all $n\geq 1$. Let $\gamma:S^n\longrightarrow Y$ 
be a representative 
of an element of $\pi_n(Y, x_0)$. Then since $S^n$ is compact, 
the image $i(S^n)$ is contained in a finite subcomplex 
$F$ of $S$ (\cite[Proposition A. 1]{hat}).  We can choose a finite subcomplex 
$\widetilde{F}$ as in Lemma \ref{lem:finite}. 
Then $\widetilde{M}:=M\cap(\widetilde{F}\times\widetilde{F})$ defines an 
acyclic matching on $\widetilde{F}$ whose set of critical simplices is 
$\widetilde{F}\cap S_c$. By Theorem \ref{thm:morse} (a), $|S_c\cap \widetilde{F}|$ 
is a deformation retract of $|\widetilde{F}|$. Therefore, $[\gamma]\in\pi_n(Y, x_0)$ is contained 
in the image of $i_*:\pi_n(X, x_0)\longrightarrow\pi_n(Y, x_0)$. 
Hence $i_*$ is surjective. 
The injectivity of $i_*$ and the connectivity of $X$ can be similarly proved. 
\end{proof}

\begin{example}
Let us consider acyclic matchings in Example \ref{ex:infinitematching}. The matching $M_1$ has 
the unique critical cell $\{\sigma_0\}$. While there are no critical cells in $M_2$. 
\end{example}

\begin{definition}
Let $P$ be a poset. The \emph{order complex} $\Delta P$ of $P$ is 
defined as 
\[
\Delta P=\{\{x_0, x_1, \dots, x_n\}\mid n\in\Z_{\geq 0}, x_i\in P, x_0<x_1<\cdots < x_n\}. 
\]
\end{definition}

\begin{remark}[On the empty simplicial complex and the void, used in \S \ref{sec:var}]
\label{rem:void}
Let $S$ be a simplicial complex on the vertex set $V$. In this paper, 
the emptyset $\emptyset\subset V$ is considered as a simplex of dimension $-1$. 
The empty simplex 
$\emptyset$ is contained in every simplex. The simplicial complex consisting of only the 
empty simplex is called the \emph{empty simplicial complex} and denoted by $\emptycpx$. 
The empty simplicial complex is a subcomplex of any simplicial complex (other than the void). 
The simplicial complex that has no simplices is called the \emph{void} and denoted by $\void$. 

For a given simplicial complex $X$, we can associate a chain complex $C_*(X): \cdots\to
C_1(X)\to C_0(X)\to C_{-1}(X)\to 0$, where $C_{-1}(X)\cong\Z$ is generated by the 
empty simplex $\emptyset$. It is nothing but the reduced chain complex in the usual sense. 
Note that $C_*(\emptycpx): \cdots\to 0\to C_{-1}(\emptycpx)=\Z\to 0\to\cdots$ 
is the chain complex supported on the degree $-1$ and $C_*(\void)=0$. 

A pair $A\subset X$ of simplicial complexes determines the chain complex of the pair 
$C_*(X, A):=C_*(X)/C_*(A)$.  Then $C_*(X, \void)=C_*(X)$ and $C_*(X, \emptycpx)$ is 
isomorphic to the usual chain complex (without empty simplex) associated with $X$. 

Recall the convention $X/\emptyset=X\sqcup\{*_0\}$, where $*_0$ is the base 
point in the category of pointed topological spaces. 
We can regard this emptyset 
as the empty simplicial complex $\emptycpx$. 
Indeed, since the chain complex 
\[
C_*(*_0): \cdots\to 0\to\Z\stackrel{1}{\longrightarrow}\Z\to0\to\cdots
\]
is homotopy equivalent to the zero chain complex, 
$C_*(X\sqcup\{*_0\})$ is homotopy 
equivalent to $C_*(X, \emptycpx)$. 

Recall that the reduced suspension $\Sus(X)$ of a pointed  space 
$X$ is defined as 
the smash product $S^1\wedge X$ with $S^1$. 
The suspension of the smash product 
of two CW complexes is known to be homotopy equivalent to 
the (unreduced) join of them 
(see \cite[Proof of Proposition 4I.1]{hat}), 
\begin{equation}
\label{eq:joinsusp}
X*Y\simeq\Sus(X\wedge Y), 
\end{equation}
which will be used in \S \ref{sec:frame}. 

Let us denote by 
\begin{equation}
\label{eq:equcone}
\Cone_\alpha(X)=\{\alpha\}*X
\end{equation}
the cone of $X$ with apex $\alpha$. 
The homotopy equivalences 
\begin{equation}
\label{eq:susp}
\begin{split}
(\{\alpha\}*X)/X&\simeq\Sus(X)\\
(X/X')*Y&\simeq (X*Y)/(X'*Y), 
\end{split}
\end{equation}
(where $X$ and $Y$ are CW complexes and $X'$ is a 
subcomplex of $X$) 
will also be frequently used in \S \ref{sec:var} and \S \ref{sec:frame}. 

In \S \ref{sec:var}, we will need the notion of (reduced) 
suspension $\Sus(X, A)$ 
of the pair of spaces. 
At the level of pairs of spaces, 
the suspension is just the product with the pair 
$(S^1, *)$ (or with $([0, 1], \{0, 1\})$). 
However, for the purpose of dealing with $\void$ and $\emptycpx$, 
the following definition is suitable (Figure \ref{fig:cones}). 
\begin{equation}
\label{eq:cone}
\Sus(X, A):=(\Cone_\alpha(X), \Cone_\alpha(A)\cup X). 
\end{equation}
Note that $\Cone_\alpha(\void)=\void$. Hence, for a pointed 
CW complex $X$, we have 
\begin{equation}
\label{eq:pairXvoid}
\Sus(X, \void):=(\Cone_\alpha(X), X),  
\end{equation}
with the quotient $\Cone_\alpha(X)/X$ homotopy equivalent to 
$\Sus(X)$. 
Since 
$\Cone_\alpha(\emptycpx)=\{\alpha\}$, 
we also note that 
\begin{equation}
\label{eq:emptyvoid}
\Sus(\emptycpx, \void):=(\{\alpha\}, \emptycpx). 
\end{equation}
Thus the corresponding chain complexes are as follows. 
\[
\begin{split}
C_*(\emptycpx, \void)&=\cdots\to 0\to C_{-1}=\Z\to 0\to\cdots, \\
C_*\left(\Sus(\emptycpx, \void)\right)&=C_*(\{\alpha\}, \emptycpx)\\
&=\cdots\to 0\to C_{0}=\Z\to 0\to\cdots. 
\end{split}
\]
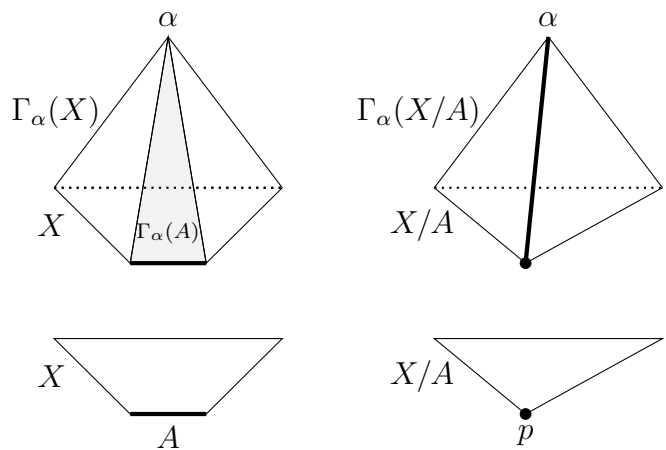
\begin{figure}[htbp]
\centering
\begin{tikzpicture}


\draw (0,1)-- node[left=0.2cm] {$X$} (1,0)--(2,0)--(3,1)--cycle; 
\draw (0,3)-- node[left=0.2cm] {$X$} (1,2)--(2,2)--(3,3); 
\draw[thick, dotted] (0,3)--(3,3); 
\draw[ultra thick] (1,0)-- node[below]{$A$}  (2,0); 
\filldraw[fill=gray!10!white, draw=black] (1,2)--(2,2)--(1.5,5)--cycle; 
\draw[ultra thick] (1,2)-- node[above=0.1cm] {\scriptsize $\Cone_\alpha(A)$} (2,2); 
\draw[thick, dotted] (0,3)--(3,3); 

\draw (0,3)-- node[left] {$\Cone_\alpha(X)$} (1.5,5) node[above] {$\alpha$};
\draw (1,2)--(1.5,5);
\draw (2,2)--(1.5,5);
\draw (3,3)--(1.5,5);

\draw (5,1)-- node[left=0.2cm] {$X/A$} (6.2,0)--(8,1)--cycle; 
\filldraw[fill=black, draw=black] (6.2, 0) node[below] {$p$} circle (2pt); ; 
\draw (5,3)-- node[left=0.2cm] {$X/A$} (6.2,2)--(8,3); 
\draw[thick, dotted] (5,3)--(8,3); 

\draw (5,3)-- node[left] {$\Cone_\alpha(X/A)$} (6.5,5) node[above] {$\alpha$};
\filldraw[fill=black, draw=black] (6.2, 2) circle (2pt); ; 
\draw[ultra thick] (6.2,2)--(6.5,5);
\draw (8,3)--(6.5,5);
\end{tikzpicture}
\caption{Quotient, cone, and reduced suspension}
\label{fig:cones}
\end{figure}
\end{remark}

\section{Causal order on metric spaces}
\label{sec:causal}

Recall that two points in the Minkowski spacetime 
$(x, t), (x', t')\in\R^{n,1}=\R^n\times\R$, with $t<t'$, 
are said to be 
\begin{itemize}
\item 
\emph{time-like} if $|x'-x|< t'-t$, 
\item 
\emph{light-like} if $|x'-x|= t'-t$, 
\item 
\emph{space-like} if $|x'-x|> t'-t$, 
\end{itemize}
(Figure \ref{fig:timelightspace}) 
where $|x-x'|$ is the Euclidean metric in $\R^n$ (\cite{sta}). 

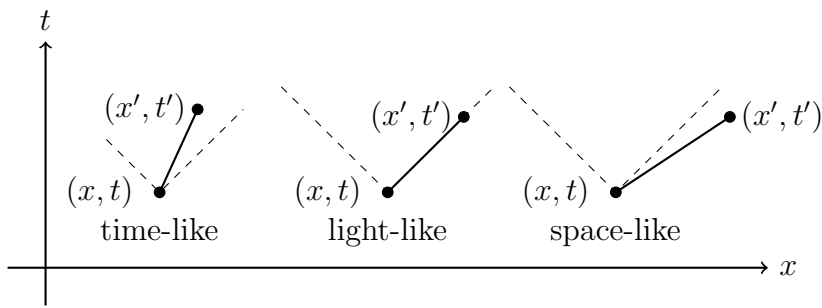
\begin{figure}[htbp]
\centering
\begin{tikzpicture}


\filldraw[fill=black, draw=black] (1,0) 
node[left=0.2cm]{$(x, t)$} node[below=0.2cm]{time-like} circle (2pt) ;
\draw[dashed] (0.3, 0.7)--(1,0)--(2.1,1.1);
\draw[thick] (1,0)--(1.5,1.1) node[left] {$(x', t')$};
\filldraw[fill=black, draw=black] (1.5,1.1) circle (2pt) ;

\filldraw[fill=black, draw=black] (4,0) 
node[left=0.2cm]{$(x, t)$} node[below=0.2cm]{light-like} circle (2pt) ;
\draw[dashed] (2.6, 1.4)--(4,0)--(5.4,1.4);
\filldraw[fill=black, draw=black] (5,1) node[left]{$(x', t')$} circle (2pt) ;
\draw[thick] (4,0)--(5,1);

\filldraw[fill=black, draw=black] (7,0) 
node[left=0.2cm]{$(x, t)$} node[below=0.2cm]{space-like} circle (2pt) ;
\draw[dashed] (5.6, 1.4)--(7,0)--(8.4,1.4);
\filldraw[fill=black, draw=black] (8.5,1) node[right]{$(x', t')$} circle (2pt) ;
\draw[thick] (7,0)--(8.5,1);

\draw[->, thick] (-1,-1)--(9,-1) node[right]{$x$}; 
\draw[->, thick] (-0.5,-1.5)--(-0.5,2) node[above]{$t$}; 

\end{tikzpicture}
\caption{Time-like-, light-like, space-like-relations}
\label{fig:timelightspace}
\end{figure}

Let $X$ be a metric space. We consider the similar structure on 
the spacetime $X\times\R$ and define a partial order using time-like and 
light-like relations. 
\begin{definition}
The \emph{causal order} (or \emph{time-light-like order}) on 
$X\times \R$ is defined by 
\begin{equation}
\label{eq:causal}
(x, t)\leq (x', t'):\Longleftrightarrow d(x, x')\leq t'-t. 
\end{equation}
\end{definition}

The relation (\ref{eq:causal}) is a metric space analogue of 
the time- or light-like relation. 
Intuitively, 
$(x, t)\leq (x', t')$ if and only if a signal (traveling at most the speed of light $c=1$)
can reach from $(x, t)$ to $(x', t')$. 

Let $a, b\in X$ and $\ell\geq 0$. We denote the causal interval 
$\{(x, t)\mid (a, 0)\leq (x, t)\leq (b, \ell)\}$ 
between 
$(a, 0)$ and $(b, \ell)\in X\times\R$ by $\Cau^\ell(X; a, b)$. More 
precisely, 
\begin{equation}
\Cau^\ell(X; a, b):=
\{(x, t)\in X\times[0, \ell]\mid 
d(a, x)\leq t \mbox{ and }
d(x, b)\leq \ell-t
\}. 
\end{equation}
We call $\Cau^\ell(X; a, b)$ the \emph{causal interval  
between $(a, 0)$ and $(b, \ell)$}, or simply, \emph{causal poset} 
(Figure \ref{fig:causal}). 
We also define the total causal poset as 
\begin{equation}
\Cau^\ell(X)=\bigsqcup_{a, b\in X}\Cau^\ell(X; a, b). 
\end{equation}

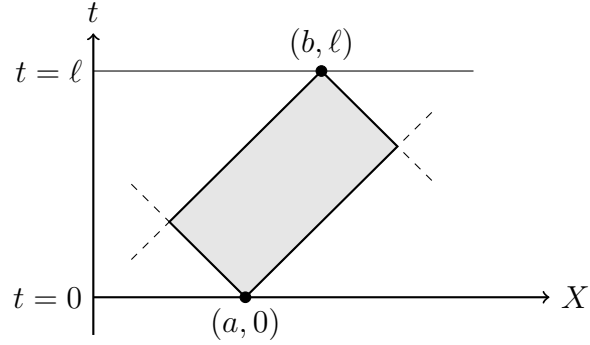
\begin{figure}[htbp]
\centering
\begin{tikzpicture}


\draw[->, thick] (0,0)--(6,0) node[right]{$X$}; 
\draw[->, thick] (0,-0.5)--(0,3.5) node[above]{$t$}; 

\filldraw[fill=gray!20!white, draw=black, thick] 
(2,0)--(1,1)--(3,3)--(4,2)--cycle;
\draw (2,0) node[below]{$(a, 0)$}; 
\draw (3,3) node[above]{$(b, \ell)$}; 
\draw[thin] (0,0) node[left]{$t=0$}; 
\draw[thin] (0,3) node[left]{$t=\ell$} --(5,3); 

\filldraw[fill=black, draw=black] (2,0) circle (2pt) ;
\filldraw[fill=black, draw=black] (3,3) circle (2pt) ;

\draw[dashed] (0.5, 1.5)--(2,0)--(4.5,2.5);
\draw[dashed] (0.5, 0.5)--(3,3)--(4.5,1.5);

\end{tikzpicture}
\caption{Causal interval $\Cau^\ell(X; a, b)$}
\label{fig:causal}
\end{figure}

We immediately have the following. 
\begin{proposition}
\begin{itemize}
\item[(1)] 
The causal poset $\Cau^\ell(X; a, b)$ is nonempty if and only if $d(a, b)\leq\ell$. 
\item[(2)] 
If $\Cau^\ell(X; a, b)\neq\emptyset$, the causal poset has the minimum 
$\min\Cau^\ell(X; a, b)=(a, 0)$ and the maximum $\max\Cau^\ell(X; a, b)=(b, \ell)$. 
\end{itemize}
\end{proposition}

\begin{example}
\label{ex:ellipse}
Let $X=\R^2$ be the Euclidean plane. Let $a, b\in X$ and $\ell\in\R_{>0}$ with 
$d(a, b)<\ell$. Then 
\[
\Cau^\ell(\R^2; a, b)=\{(x, t)\in\R^2\times[0, \ell] 
\mid d(a, x)\leq t \mbox{ and }d(x, b)\leq \ell-t\}. 
\]
The projection of the causal interval to $\R^2$ is 
\[
\{x\in\R^2\mid d(a, x)+d(x, b)\leq \ell\}
\]
which is an ellipse (Figure \ref{fig:ellipse}). 
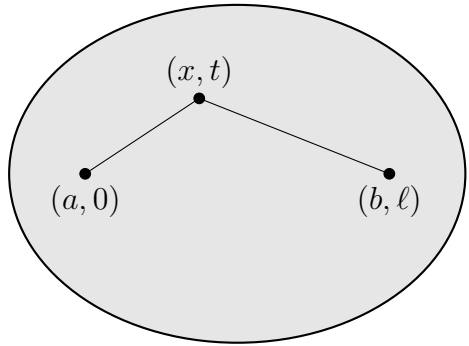
\begin{figure}[htbp]
\centering
\begin{tikzpicture}


\filldraw[fill=gray!20!white, draw=black, thick] (0,0) ellipse (3 and 2.24); 

\filldraw[fill=black, draw=black] (2,0) circle (2pt) node[below]{$(b, \ell)$};
\filldraw[fill=black, draw=black] (-2,0) circle (2pt) node[below]{$(a, 0)$};
\draw[black] (-2,0)--(-0.5,1)--(2,0);
\filldraw[fill=black, draw=black] (-0.5,1) circle (2pt) node[above]{$(x, t)$};

\end{tikzpicture}
\caption{The projection of $\Cau^\ell(\R^2; a, b)$ (Example \ref{ex:ellipse}).}
\label{fig:ellipse}
\end{figure}
\end{example}

\begin{example}
\label{ex:2pts}
Even if $X$ is a finite metric space, $\Cau^\ell(X; a, b)$ may not be a finite poset. 
Let $X=\{a, b\}$ be a metric space consisting of two points with distance $d(a, b)=1$. 
Then 
\begin{equation}
\Cau^\ell(X; a, b)=
\begin{cases}
\emptyset, &\mbox{ if }\ell<1, \\
\{(a, 0), (b, 1)\}, &\mbox{ if }\ell=1, \\
\{(a, t)\mid 0\leq t\leq \ell-1\}\cup\{(b, t')\mid , 1\leq t'\leq \ell\}, &\mbox{ if }\ell>1. 
\end{cases}
\end{equation}
See also Example \ref{ex:Cau} for more examples. 
\end{example}

\section{Magnitude homotopy type}
\label{sec:mht}

\subsection{Construction}
\label{sec:const}

Let $X$ be a metric space, $a, b\in X$ and $\ell\geq 0$. In this section, 
we consider the order complex $\Delta\Cau^\ell(X; a, b)$ of the 
causal poset $\Cau^\ell(X; a, b)$. By definition, $\Delta\Cau^\ell(X; a, b)$ 
is the simplicial complex consisting of sequences 
with time parameters 
\begin{equation}
\label{eq:TLseq}
((x_0, t_0), (x_1, t_1), \dots, (x_n, t_n)), 
\end{equation}
of elements in $X\times [0, \ell]$ satisfying 
\begin{equation}
\begin{split}
\label{eq:ineq}
d(a, x_0)&\leq t_0, \\
d(x_{i-1}, x_i)&\leq t_i-t_{i-1}\ (i=1, \dots, n), \\
d(x_n, b)&\leq\ell-t_n. 
\end{split}
\end{equation}
We call such a sequence \emph{causal sequence}. 
Causal sequences of the following type will play an important role. 

\begin{definition}
If the causal sequence (\ref{eq:TLseq}) satisfies 
$t_0=0, t_n=\ell$ and $d(x_{i-1}, x_i)=t_i-t_{i-1}$ for all $1\leq i\leq n$, 
then it is called a \emph{light-like sequence} of length $\ell$ from point $a$ to $b$.   
(Figure \ref{fig:sequences}). 
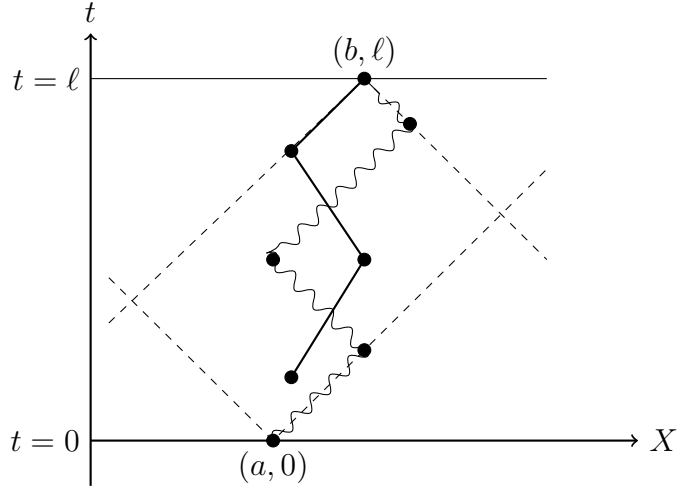
\begin{figure}[htbp]
\centering
\begin{tikzpicture}[scale=1.2]


\draw[->, thick] (0,0)--(6,0) node[right]{$X$}; 
\draw[->, thick] (0,-0.5)--(0,4.5) node[above]{$t$}; 

\draw (2,0) node[below]{$(a, 0)$}; 
\draw (3,4) node[above]{$(b, \ell)$}; 
\draw[thin] (0,4) node[left]{$t=\ell$} --(5,4); 
\draw[thin] (0,0) node[left]{$t=0$}; 

\filldraw[fill=black, draw=black] (2,0) circle (2pt) ;
\filldraw[fill=black, draw=black] (3,4) circle (2pt) ;

\draw[dashed] (0.2, 1.8)--(2,0)--(5,3);
\draw[dashed] (0.2, 1.3)--(3,4)--(5,2);

\draw[decorate,decoration={coil,aspect=0}] (2,0)  -- (3,1) -- (2,2) -- (3.5,3.5) -- (3,4);
\draw[thick] (2.2, 0.7)--(3,2)--(2.2,3.2)--(3,4);

\filldraw[fill=black, draw=black] (3,1) circle (2pt) ;
\filldraw[fill=black, draw=black] (2,2) circle (2pt) ;
\filldraw[fill=black, draw=black] (3.5,3.5) circle (2pt) ;
\filldraw[fill=black, draw=black] (2.2,0.7) circle (2pt) ;
\filldraw[fill=black, draw=black] (3,2) circle (2pt) ;
\filldraw[fill=black, draw=black] (2.2,3.2) circle (2pt) ;

\end{tikzpicture}
\caption{A general causal sequence (straight segments) and a light-like one (wavy segments)}
\label{fig:sequences}
\end{figure}
\end{definition}

\begin{proposition}
\label{prop:can}
A causal sequence (\ref{eq:TLseq}) is light-like if and only if 
$d(x_0, \dots, x_n)=\ell$. 
\end{proposition}
\begin{proof}
If $d(x_0, \dots, x_n)=\ell$, then by (\ref{eq:ineq}), we have 
\[
\ell=d(x_0, \dots, x_n)=\sum_{i=1}^nd(x_{i-1}, x_i)\leq t_n-t_0. 
\]
Since $t_n\leq\ell$, we have $t_0=0, t_n=\ell$ and $d(x_{i-1}, x_i)=t_i-t_{i-1}$ for 
all $1\leq i\leq n$. 

The converse is straightforward. 
\end{proof}

If a sequence (\ref{eq:TLseq}) is light-like, then $t_k$ is 
expressed as $t_k=\sum_{i=1}^k d(x_{i-1}, x_i)$ for $k\geq 1$ (and $t_0=0$). 
So the parameter $t_k$ are recovered from the points $x_0, \dots, x_n$. 
We will sometimes omit the parameter $t_k$ and denote the light-like 
sequence $(x_0, \dots, x_n)$. 

Since $d(x_0, \dots, x_n)<\ell$ implies $d(x_0, \dots, \widehat{x_k}, \dots, x_n)<\ell$, 
non-light-like sequences determine a subcomplex of the order complex 
$\Delta\Cau^\ell(X; a, b)$. 

\begin{definition}
\label{def:Deltadot}
$\Deltadot\Cau^\ell(X; a, b)$ is the subcomplex of $\Delta\Cau^\ell(X; a, b)$ 
consisting of causal sequence $\sigma=((x_0, t_0), \dots, (x_n, t_n))$ satisfying 
\[
d(x_0,  \dots, x_n)<\ell. 
\]
$\Deltadot\Cau^\ell(X)$ is also similarly define. 
\end{definition}

\begin{definition}
\label{def:mghtpy}
The \emph{magnitude homotopy type} for $a, b\in X$ and $\ell\in\R_{\geq 0}$ is defined 
as the pointed CW-complex 
\begin{equation}
\calM^\ell(X; a, b):=
\frac{|\Delta\Cau^\ell(X; a, b)|}{|\Deltadot\Cau^\ell(X; a, b)|}. 
\end{equation}
We also define the total magnitude homotopy type by 
\[
\calM^\ell(X)
=
\frac{|\Delta\Cau^\ell(X)|}{|\Deltadot\Cau^\ell(X)|}. 
\]
\end{definition}

Note that each $\calM^\ell(X; a, b)$ is a pointed space. The total magnitude homotopy 
type also has the following expression in terms of wedge sum 
\[
\calM^\ell(X)
=\bigvee_{a, b\in X}\calM^\ell(X; a, b). 
\]

\begin{remark}
The space $\calM^\ell(X)$ was constructed 
by Hepworth-Willerton \cite[Definition 8.1]{HW} for graphs and 
by Bottinelli-Kaiser \cite[Definition 4.4]{bot-kai} for metric spaces 
as the realization of certain simplicial set. 
Our Definition \ref{def:mghtpy} realizes it as the quotient of a 
simplicial complex by a subcomplex. 
\end{remark}

\begin{definition}
Two metric spaces $X$ and $Y$ are said to be \emph{magnitude homotopy equivalent} if 
$\calM^\ell(X)$ and $\calM^\ell(Y)$ are homotopy equivalent for any $\ell\geq 0$. 
\end{definition}

\begin{theorem}
\label{thm:mainisom}
For $k\in\Z_{\geq 0}$, $\ell\in\R_{\geq 0}$ and $a, b\in X$, we have 
\begin{equation}
\begin{split}
\widetilde{H}_k(\calM^\ell(X; a, b))&\cong\MH^\ell_k(X; a, b),\\
\widetilde{H}_k(\calM^\ell(X))&\cong\MH^\ell_k(X). 
\end{split}
\end{equation}
\end{theorem}
\begin{proof}
First note that the reduced homology group of the quotient space 
$\widetilde{H}_n(\calM^\ell(X; a, b))$ is isomorphic to 
the homology group of the pair 
$H_*(\Delta\Cau^\ell(X; a, b), \Deltadot\Cau^\ell(X; a, b))$. 
We compare the magnitude chain complex $\MC^\ell_*(X; a, b)$ with 
the chain complex of the pair 
$C_*(\Delta\Cau^\ell(X; a, b), \Deltadot\Cau^\ell(X; a, b))$. 
The magnitude chain complex $\MC^\ell_n(X; a, b)$ is 
generated by sequences 
$(x_0, \dots, x_n)$ with $x_0=a, x_n=b$ and $d(x_0, \dots, x_n)=\ell$. 
Let $t_i=d(a, x_0, \dots, x_i)$, for $i=0, \dots, n$. 
Then 
we obtain a light-like sequence 
\[
((x_0, t_0), (x_1, t_1), \dots, (x_n, t_n)). 
\]
This correspondence gives a chain map 
\begin{equation}
\label{eq:chainisom}
\MC^\ell_*(X; a, b)\longrightarrow 
C_*(\Delta\Cau^\ell(X; a, b), \Deltadot\Cau^\ell(X; a, b)). 
\end{equation}
To verify this fact, consider the $i$-th boundary 
$\partial_i(x_0, \dots, x_n)$ (see Definition \ref{def:MC}). 
If $d(x_{i-1}, x_{i+1})=d(x_{i-1}, x_i, x_{i+1})$, then it is also equal to 
$t_{i+1}-t_{i-1}$. Hence, the causal sequence 
\begin{equation}
\label{eq:causalbdry}
((x_0, t_0), \dots, (x_{i-1}, t_{i-1}), (x_{i+1}, t_{i+1}), \dots, (x_n, t_n))
\end{equation}
is also a light-like sequence, which is equal to the image of 
$\partial_i(x_0, \dots, x_n)$. 
If $d(x_{i-1}, x_{i+1})<d(x_{i-1}, x_i, x_{i+1})$, then 
$\partial_i(x_0, \dots, x_n)=0$, and since 
$d(x_{i-1}, x_{i+1})<t_{i+1}-t_{i-1}$, the sequence 
(\ref{eq:causalbdry}) is not light-like and 
is contained in $\Deltadot\Cau^\ell(X; a, b)$. 

Clearly (\ref{eq:chainisom}) gives an injective chain map. 
We will prove the surjectivity. 
Recall that the relative complex 
$C_*(\Delta\Cau^\ell(X; a, b), \Deltadot\Cau^\ell(X; a, b))$ 
is generated by simplices which is not contained in 
$\Deltadot\Cau^\ell(X; a, b)$. 
Let $\sigma=((x_0, t_0), \dots, (x_n, t_n))$ be an $n$-simplex in 
$\Delta\Cau^\ell(X; a, b)$. Then, by Proposition \ref{prop:can}, 
$\sigma$ is not contained in $\Deltadot\Cau^\ell(X; a, b)$ (equivalently 
$d(x_0, \dots, x_n)=\ell$)  if and only if 
it is a light-like sequence with $x_0=a, x_n=b$ and $d(x_0, \dots, x_n)=\ell$. 
This is clearly obtained as an image of the above map from $\MC^\ell_n(X; a, b)$. 
Hence  (\ref{eq:chainisom}) is an isomorphism of 
chain complexes. 
\end{proof}

\begin{remark}
\label{rem:empty}
(on $\ell=0$) 
If $\Cau^\ell(X; a, b)=\emptyset$, then $\calM^\ell(X; a, b)=\{*_0\}$. 
In particular, we consider the case $\ell=0$. 
If $a=b$, then $\Cau^0(X; a, a)=\{(a, 0)\}$, otherwise, $\Cau^0(X; a, b)=\emptyset$. 
We denote the set $\{(a, 0)\}$ by $*_1$. Then 
$\Delta\Cau^0(X; a, a)=\{*_1\}$ and, 
$\Deltadot\Cau^0(X; a, a)=\emptyset$, 
we have 
\[
\calM^0(X; a, b)=
\begin{cases}
\{*_1\}/\emptyset=\{*_1, *_0\}, & a=b, \\
\emptyset/\emptyset=
\{*_0\}, & a\neq b. 
\end{cases}
\]
\end{remark}

If $\ell<d(a, b)$, then $\Cau^\ell(X; a, b)=\emptyset$. 
We skip this case (we will discuss in \S \ref{sec:var}). 
Suppose $\ell=d(a, b)>0$. 
Then by definition, 
$\Cau^{d(a, b)}(X; a, b)$ is the set of $(x, t)\in X\times [0, d(a, b)]$ such that 
\[
d(a, x)\leq t,\ \ d(x, b)\leq d(a, b)-t. 
\]
By the triangle inequality, we have $t=d(a, x)$ and $d(a, b)=d(a, x, b)$. This is nothing but the 
so-called \emph{interval} in the metric space. We here introduce some kinds of intervals, 
which are analogy of closed and open intervals: For $a, b\in X$, 
\begin{equation}
\begin{split}
I[a, b]&:=\{x\in X\mid d(a, x, b)=d(a, b)\}, \\
I(a, b)&:= I[a, b]\smallsetminus\{a, b\}, \\
I(a, b]&:=I[a, b]\smallsetminus\{a\}, \\
I[a, b)&:=I[a, b]\smallsetminus\{b\}. 
\end{split}
\end{equation}
These intervals have natural order structure defined by the relation 
``$x\leq y$ if and only if $d(a, x, y, b)=d(a, b)$''. Note that 
$\Cau^{d(a, b)}(X; a, b)\cong I[a, b]$ 
as posets. 
Using these interval posets, 
the magnitude homotopy type for the case $\ell=d(a, b)$ 
can be expressed as follows. 
\begin{proposition}
\label{prop:kernel}
Let $X$ be a metric space and $a, b\in X$. Then 
\begin{equation}
\label{eq:interv}
\calM^{d(a, b)}(X; a, b)\approx
\frac{|\Delta I[a, b]|}{|\{\sigma\in\Delta I[a, b]\mid 
\sigma\not\supseteq\{a, b\}\}|}. 
\end{equation}
\end{proposition}
\begin{proof}
First note that the correspondence 
$I[a, b]\ni x\longmapsto (x, d(a, x))\in
\Cau^{d(a, b)}(X; a, b)$ gives an isomorphism 
$I[a, b]\cong \Cau^{d(a, b)}(X; a, b)$. 
Thus we have 
$\Delta\Cau^{d(a, b)}(X; a, b)\cong\Delta I[a, b]$. 
Let $\bm{x}=(x_0, \dots, x_n)\in\Delta I[a, b]$. 
Then $d(\bm{x})=d(a, b)$ if and only if 
$x_0=a, x_n=b$. 
Hence the sequence becomes shorter if and only if 
the sequence does not contain $\{a, b\}$. 
\end{proof}

\begin{example}
\label{ex:easy}
\begin{itemize}
\item[(1)] 
Let $C_4$ be the cycle graph with four vertices 
as in Figure \ref{fig:firstexamples}. 
Then $I[a, b]=\{a, b, c, d\}$ with order relations 
$a\leq d\leq b$ and $a\leq c\leq d$. The 
order complex $\Delta I[a, b]$ is a union of two simplices 
$\{a, d, b\}$ and $\{a, c, b\}$. 
The boundary edges $\{a, d\}, \{d, b\}, \{a, c\}, \{c, b\}$ 
are shorter than $2=d(a, b)$. 
Hence $\calM^2(C_4; a, b)\simeq S^2$. 
\begin{figure}[htbp]
\centering
\begin{tikzpicture}


\filldraw[thick] (0,0) circle (2pt) node[left] {$a$} -- (2,0) circle (2pt) node[right] {$c$} -- (2,2) circle (2pt) node[right] {$b$} -- (0,2) circle (2pt) node[left] {$d$} -- (0,0); 

\filldraw[fill=gray!20!white, draw=black, thick] 
(7,0) node[below] {$a$} --(5,1) node[left] {$d$} --(7,2) node[above] {$b$} --(9,1)node[right] {$c$} --(7,0)--(7,2);

\end{tikzpicture}
\caption{The cycle graph $C_4$ and $\Delta I[a, b]$.}
\label{fig:firstexamples}
\end{figure}
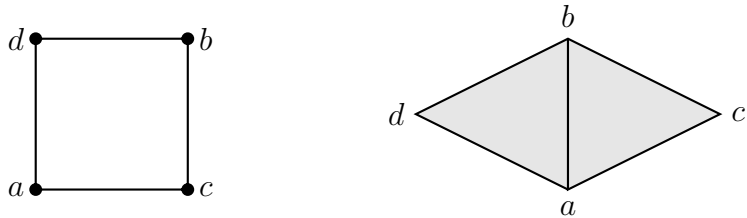
\item[(2)] Let $X$ be a metric space and $a, b\in X$. 
Suppose that the interval $I[a, b]$ is 
totally ordered. 
Then the numerator $\Delta I[a, b]$ of the formula (\ref{eq:interv})
is a (possibly infinite dimensional) simplex. The denominator 
is the union of simplices which does not contain the $1$-simplex 
$\{a, b\}$. Thus we have 
\[
\calM^{d(a, b)}(X; a, b)\simeq
\begin{cases}
*, &\mbox{ if } I(a, b)\neq\emptyset, \\
S^1,  &\mbox{ if } I(a, b)=\emptyset. 
\end{cases}
\]
\end{itemize}
\end{example}

Next we present several examples of total magnitude homotopy types for graphs. 
\begin{example}
\label{ex:graphs}
Examples of the magnitude homotopy type for several finite graphs 
$G=(V, E)$. 
\begin{itemize}
\item[(1)] 
Let $G$ be a tree. Then, 
\[
\calM^\ell(G)\approx
\begin{cases}
S^0\vee\cdots\vee S^0, (\text{wedge of }|V|\text{ spheres})\approx\{(|V|+1) \text{ points}\}, &\ell=0, \\
S^\ell\vee\cdots\vee S^\ell, (\text{wedge of }2|E|\text{ spheres}), &\ell\geq 1.
\end{cases}
\]
(This is not trivial. 
We will prove more general results in \S \ref{sec:frame}. 
See Corollary \ref{cor:toset}.) 
\item[(2)]Let $G$ be a complete graph with $m$ vertices. Then, 
\[
\calM^\ell(G)\approx S^\ell\vee\cdots\vee S^\ell, 
(\text{ wedge of }m(m-1)^\ell\text{ spheres}).
\]
Note that the degree and length coincide for any sequence of 
the complete graph. So, the boundary sequence has 
shorter length than the original sequence. Therefore, 
the maximal faces of $\Delta\Cau^\ell(G)$ are all 
$\ell$-dimensional simplices, and the boundary simplices 
are all belonging to $\Deltadot\Cau^\ell(G)$. 
\item[(3)]\cite[Theorem 3.4]{taj-yos} Let $G$ be a pawful graph.
Recall that a pawful graph $G$ is a graph satisfying 
the following conditions:
\begin{itemize}
\item$d(x, y)\leq 2$ for any $x, y\in G$, 
\item for any $x, y, z\in G$ with $d(x, y)=d(y, z)=2$ and 
$d(x, z)=1$, there exists $a\in G$ such that $d(a, x)=d(a, y)=d(a, z)=1$. 
\end{itemize}
Then, $\calM^\ell(G)$ is homotopy equivalent to wedge of 
$\ell$-spheres. 
(It is proved that the Asao-Izumihara complex 
(defined in \S \ref{sec:var}) is homotopy equivalent to 
a wedge of $(\ell-2)$-spheres for pawful graphs in \cite{taj-yos}. 
We will prove later (Theorem \ref{thm:double}) 
that the magnitude homotopy type is homotopy equivalent to 
the double suspension of the Asao-Izumihara complex.) 
\end{itemize}
So far magnitude homotopy types are always wedge of spheres. 
However, in general, the magnitude homotopy type of 
a graph can become arbitrarily complicated. 
(See Corollary \ref{cor:arbitrary}, Proposition \ref{prop:const}.) 
\end{example}

\begin{remark}
It is natural to ask whether properties of magnitude homology groups 
can be explained by classical topological results on the magnitude homotopy types. 
We will see that K\"unneth formula and Mayer-Vietoris type formula for 
magnitude homology groups are actually explained by the notions of 
classical topology of CW complexes via magnitude homotopy types. 
\end{remark}

\begin{remark}
There are several variants of the magnitude homotopy type. 
\begin{itemize}
\item[(i)] 
Let $\varepsilon>0$. Define $\Delta^{\leq\ell-\varepsilon}\Cau^\ell(X; a, b)$ by 
\[
\Delta^{\leq\ell-\varepsilon}\Cau^\ell(X; a, b)=
\{((x_0, t_0), \dots, (x_n, t_n))\in\Delta\Cau^\ell(X; a, b)\mid 
d(x_0, \dots, x_n)\leq\ell-\varepsilon\}. 
\]
This is clearly a subcomplex of $\Delta\Cau^\ell(X; a, b)$. 
Let 
\[
\calM^{(\ell-\varepsilon, \ell]}(X; a, b):= 
\frac{|\Delta\Cau^\ell(X; a, b)|}{|\Delta^{\leq\ell-\varepsilon}
\Cau^{\ell}(X; a, b)|}. 
\]
The reduced homology group of this space captures 
a variants of magnitude homology groups. 
Namely, the chain is generated by time-parametrized sequence 
$((x_0, t_0), \dots, (x_n, t_n))$ 
satisfying 
\[
d(x_{i-1}, x_i)\leq t_i-t_{i-1}, \mbox{ for }i=0, \dots, n+1, 
\]
where $x_{-1}=a, x_{n+1}=b, t_{-1}=0, t_{n+1}=\ell$ and 
\[
\ell-\varepsilon < d(x_0, x_1, \dots, x_n)\leq\ell. 
\]
\item[(ii)] 
The notion of Lorentzian length space (\cite{ks-lor, min-lor}) 
is a generalization of the Minkowski space $\R^n\times\R$. 
We can define time-like, light-like, or causal paths for 
Lorentzian length space. 
It seems to be an interesting direction 
to study the magnitude homotopy type (or homology group) 
for Lorentzian length spaces. 
\end{itemize}
\end{remark}

\subsection{Recovering finite metric spaces}
\label{sec:recover}

Suppose $X$ is a finite metric space. Then for each 
$\ell\geq 0$ and $a, b\in X$, 
there are only finitely many light-like sequences of length $\ell$ 
from $a$ to $b$. 
Hence $\calM^\ell(X; a, b)$ and their wedge 
sum $\calM^\ell(X)$ is a finite CW complex. The magnitude 
can be expressed in terms of the reduced Euler characteristic 
$\widetilde{\chi}(-)=\sum_{k\geq 0}(-1)^k\cdot\rank\widetilde{H}_k(-)$ 
of the magnitude homotopy type $\calM^\ell(X)$. 
From the results in \S \ref{sec:MH}, we have 
the following. 
\begin{proposition}
\label{prop:inverse}
Let $X$ be a finite metric space. 
Then the matrix $Z_X^{-1}$, the magnitude weighting 
$w:X\longrightarrow\Hahn$, and the magnitude $\Mag(X)\in\Hahn$ 
are expressed as follows. 
\begin{equation}
\begin{split}
Z_X^{-1}(a, b)&=\sum_{\ell\geq 0}\widetilde{\chi}(\calM^\ell(X; a, b))q^\ell, \\
w(a)&=\sum_{\ell\geq 0, b\in X}
\widetilde{\chi}(\calM^\ell(X; a, b))q^\ell, \\
\Mag(X)&=\sum_{\ell\geq 0}\widetilde{\chi}(\calM^\ell(X))q^\ell. 
\end{split}
\end{equation}
\end{proposition}

We can also prove that the family of spaces 
$\{\calM^\ell(X; a, b)\}_{\ell\geq 0, a, b\in X}$ has information on the metric. 
More precisely, we have the following. 
\begin{theorem}
\label{thm:recover}
Let $X$ and $Y$ be finite metric spaces with $|X|=|Y|$. 
Let $f:X\longrightarrow Y$ be a map. 
Then the following are equivalent. 
\begin{itemize}
\item[$(a)$] 
The map $f$ is an isometry, i.e, 
$d_Y(f(a), f(b))=d_X(a, b)$ for any $a, b\in X$. 
\item[$(b)$] 
$\calM^\ell(X; a, b)$ and $\calM^\ell(Y; f(a), f(b))$ are homotopy equivalent for any $\ell\geq 0$, 
and $a, b\in X$. 
\item[$(c)$] 
$\widetilde{\chi}(\calM^\ell(X; a, b))=\widetilde{\chi}(\calM^\ell(Y; f(a), f(b)))$ 
for any $\ell\geq 0$ and $a, b\in X$. 
\end{itemize}
\end{theorem}

\begin{proof}
The implications $(a)\Longrightarrow (b)\Longrightarrow (c)$ is obvious. 
Assume $(c)$. By Proposition \ref{prop:inverse}, $Z_X^{-1}(a, b)=Z_Y^{-1}(f(a), f(b))$ holds 
for any $a, b\in X$. Hence $Z_X^{-1}=Z_Y^{-1}$. Taking the inverse, we have $Z_X=Z_Y$. 
This implies $d_X(a, b)=d_Y(f(a), f(b))$ for any $a, b\in X$, 
thus we have $(a)$. 
\end{proof}

\subsection{Smaller model}
\label{sec:smaller}

The cells of the magnitude homotopy type $\calM^\ell(X; a, b)$ 
are one-to-one corresponding to the light-like sequence 
$((x_0, t_0), \dots, (x_n, t_n))$. 
Since the sequence with shorter length $d(x_0, \dots, x_n)<\ell$ is 
contained in $\Deltadot\Cau^\ell(X; a, b)$, it 
does not contribute to $\calM^\ell(X; a, b)$. 
We can construct a smaller poset than $\Cau^\ell(X; a, b)$ to define $\calM^\ell(X; a, b)$. 

\begin{definition}
\label{def:essential}
A point $(x, t)\in \Cau^\ell(X; a, b)$ is called \emph{an essential point} if 
it is a point in some light-like sequence of length $\ell$ from point $a$ to $b$. 
Denote by 
$\Cau_{\ess}^\ell(X; a, b)$ the set of all essential points, which is called 
\emph{the essential subposet} of the causal poset $\Cau^\ell(X; a, b)$. 
More precisely, 
\begin{equation}
\Cau_{\ess}^\ell(X; a, b):=
\left\{
(x, t)\in\Cau^\ell(X; a, b)
\left|
\begin{array}{l}
\exists n\geq k\geq 0, \exists x_0, \dots, x_k, \dots, x_n\in X\\
\mbox{s.t. }
x=x_k, t=d(a, x_0, \dots, x_k). \\
\mbox{and }d(a, x_0, \dots, x_n, b)=\ell. 
\end{array}
\right.
\right\}. 
\end{equation}
\end{definition}

As in the case of causal order complexes, we define 
\begin{equation}
\label{eq:intersect}
\Deltadot\Cau_{\ess}^\ell(X; a, b):=
\Delta\Cau_{\ess}^\ell(X; a, b)\cap\Deltadot\Cau^\ell(X; a, b). 
\end{equation}
The poset $\Cau_{\ess}^\ell(X; a, b)$ is smaller than $\Cau^\ell(X; a, b)$, which 
can define the magnitude homotopy type as follows. 
\begin{proposition}
\[
\calM^\ell(X; a, b)\approx
\frac{|\Delta\Cau_{\ess}^\ell(X; a, b)|}{|\Deltadot\Cau_{\ess}^\ell(X; a, b)|}. 
\] 
\end{proposition}
\begin{proof}
Clearly, we have a map of pairs of simplicial complexes, 
\[
(\Delta\Cau_{\ess}^\ell(X; a, b), \Deltadot\Cau_{\ess}^\ell(X; a, b))
\longrightarrow
(\Delta\Cau^\ell(X; a, b), \Deltadot\Cau^\ell(X; a, b)). 
\]
In view of (\ref{eq:intersect}), 
it is enough to show that every light-like sequence of $\Delta\Cau^\ell(X; a, b)$ 
is contained in $\Delta\Cau_{\ess}^\ell(X; a, b)$. Let $((x_0, t_0), \dots, (x_n, t_n))$ 
be a light-like sequence of $\Delta\Cau^\ell(X; a, b)$. Then by definition, 
$(x_k, t_k)\in\Cau_{\ess}^\ell(X; a, b)$ for $k=0, 1, \dots, n$. 
Hence the original light-like sequence is a chain in $\Delta\Cau_{\ess}^\ell(X; a, b)$. 
\end{proof}

\begin{example}
\label{ex:previous} Let $X=\{a, b\}$ with $d(a, b)=1$ (as in Example \ref{ex:2pts}). 
If $\ell>1$, there are no light-like sequence from point $a$ to $b$ of length $\ell$. Hence, 
$\Cau_{\ess}^\ell(X; a, b)=\emptyset$. 
\end{example}

\begin{example}
\label{ex:Cau}
Let $X=\{a, b, c\}$ be the vertex set of the complete graph $K_3$ as in 
Figure \ref{fig:K3}. Then $\Cau^2(X; a, b)$ and $\Cau_{\ess}^\ell(X; a, b)$ 
are as follows. 
\begin{equation}
\begin{split}
\Cau^2(X; a, b)&=\{(a, t), (c, 1), (b, t')\mid 0\leq t\leq 1, 1\leq t'\leq 2\}, \\
\Cau_{\ess}^2(X; a, b)&=\{(a, 0), (b, 2), (c, 1)\}. 
\end{split}
\end{equation}
The poset structure of $\Cau_{\ess}^2(X; a, b)$ is simple. 
It is just a linear order of three elements $(a, 0)<(c, 1)<(b, 2)$. 
However, the poset structure of $\Cau^2(X; a, b)$ is not simple. 
There are several types of inequalities (see the right of Figure \ref{fig:K3}): 
\begin{itemize}
\item 
$(a, t)<(a, t')$ if $0\leq t<t'\leq 1$, 
\item 
$(b, t)<(b, t')$ if $1\leq t<t'\leq 2$, 
\item 
$(a, t)<(b, t')$ if $1+t<t'\leq 2$, in particular, $(a, t)<(b, 1+t)$ is a covering relation, 
that is, there are no elements between these two elements, 
\item 
$(a, 0)<(c, 1)<(b, 2)$. 
\end{itemize}
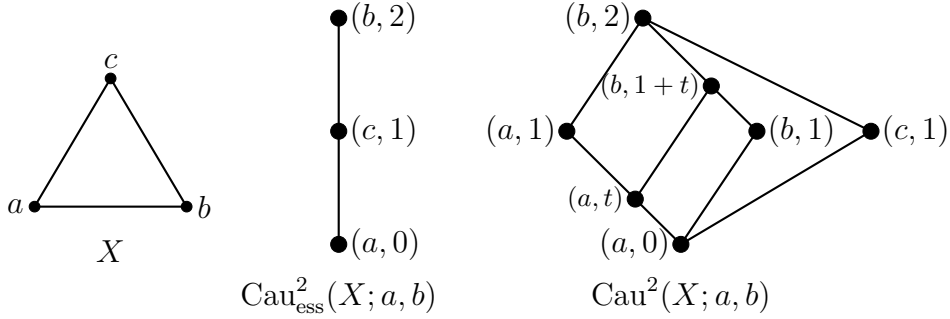
\begin{figure}[htbp]
\centering
\begin{tikzpicture}


\filldraw[fill=black, draw=black] (-2,0.5) node[left] {$a$} circle (2pt); 
\filldraw[fill=black, draw=black] (0,0.5) node[right] {$b$} circle (2pt); 
\filldraw[fill=black, draw=black] (-1,2.2) node[above] {$c$} circle (2pt); 
\draw[thick] (-2,0.5)-- node[below=0.3cm] {$X$} (0,0.5)--(-1,2.2)--cycle; 

\filldraw[fill=black, draw=black] (2,0) node[below=0.3cm] {$\Cau_{\ess}^2(X; a, b)$} 
node[right] {$(a, 0)$} circle (3pt); 
\filldraw[fill=black, draw=black] (2,1.5) node[right] {$(c, 1)$} circle (3pt); 
\filldraw[fill=black, draw=black] (2,3) node[right] {$(b, 2)$} circle (3pt); 
\draw[thick] (2, 0)--(2, 3);

\filldraw[fill=black, draw=black] (6.5,0) node[below=0.3cm] {$\Cau^2(X; a, b)$} 
node[left] {$(a, 0)$} circle (3pt); 
\filldraw[fill=black, draw=black] (5,1.5) node[left] {$(a, 1)$} circle (3pt); 
\filldraw[fill=black, draw=black] (7.5,1.5) node[right] {$(b, 1)$} circle (3pt); 
\filldraw[fill=black, draw=black] (6,3) node[left] {$(b, 2)$} circle (3pt); 
\draw[thick] (6.5, 0)--(7.5, 1.5)--(6,3)--(5, 1.5)--cycle; 

\filldraw[fill=black, draw=black] (9,1.5) node[right] {$(c, 1)$} circle (3pt); 
\draw[thick] (6.5, 0)--(9, 1.5)--(6,3); 

\filldraw[fill=black, draw=black] (5.9,0.6) circle (3pt); 

\filldraw[fill=black, draw=black] (6.9,2.1) circle (3pt); 

\draw[thick] (5.9, 0.6) node [left]{{\footnotesize $(a, t)$}} --(6.9, 2.1) node[left] {{\footnotesize $(b, 1+t)$}}; 

\end{tikzpicture}
\caption{$X=\{a, b, c\}$ and poset structure of 
$\Cau_{\ess}^\ell(X; a, b)$ and $\Cau^\ell(X; a, b)$}
\label{fig:K3}
\end{figure}
In this case, $\Delta\Cau_{\ess}^2(X; a, b)$ is the $2$-simplex with vertices 
$(a, 0), (c, 1)$ and $(b, 2)$ 
and 
$\Deltadot\Cau_{\ess}^2(X; a, b)$ is its boundary. Hence $\calM^\ell(X; a, b)\approx S^2$. 
See also Example \ref{ex:graphs}. 
\end{example}

\begin{remark}
\label{rem:advantage}
The smaller model (using the essential subsets $\Cau_{\ess}^\ell(X)$) of 
$\calM^\ell(X)$ is useful for specific computations. 
In the next section (\S \ref{sec:var}), we will also investigate the relationship between 
the magnitude homotopy type and further smaller model constructed by 
Asao-Izumihara \cite{AI} for graphs. 

There are several advantages of the bigger model 
(using $\Cau^\ell(X)$) of $\calM^\ell(X)$. 
Firstly, the definition becomes simpler. 
This advantage is not limited to its mere simplicity, 
but also holds essential significance. 
Indeed, the comparison with subspace is easier. For example, 
let $A\subset X$ be a subset of a metric space $X$. 
We consider $A$ as a metric 
space by the induced metric. Let $a, b\in A$, then, 
\begin{equation}
\Cau^\ell(X; a, b)\cap (A\times [0, \ell])=
\Cau^\ell(A; a, b) 
\end{equation}
holds. However, for $\Cau_{\ess}^\ell(X; a, b)$, 
\begin{equation}
\Cau_{\ess}^\ell(X; a, b)\cap (A\times [0, \ell])\supsetneq
\Cau_{\ess}^\ell(A; a, b), 
\end{equation}
in general (see Example \ref{ex:Cau} for explicit examples). 
This makes it more lengthy to 
write down the proof of several results 
(e.g., Mayer-Vietoris formula \S \ref{sec:mv}) 
using the smaller model.  It may be possible to compare 
the relationship between 
larger and smaller models with the relationship between 
singular chain complex of a 
topological space and finite chain complexes of 
finite triangulated space. 
\end{remark}

\subsection{Relation with Asao-Izumihara type complex}
\label{sec:var}

Asao-Izumihara \cite{AI} constructed a pair of simplicial complexes $K_\ell(G; a, b)$ 
and $K'_\ell(G; a, b)$ for a graph $G$ and $\ell\in\Z_{>0}$, and proved that 
\begin{equation}
\label{eq:AIisom}
\MH_k^\ell(X; a, b)\cong
H_{k-2}(|K_\ell(X; a, b)|, |K'_\ell(X; a, b)|) 
\end{equation}
for $k\geq 3$. They also describe the case $k=2$ by dividing into cases 
$d(a, b)=\ell$ and $d(a, b)<\ell$. In this section, we will prove that 
the isomorphism (\ref{eq:AIisom}) holds for any metric space and for $k\geq 0$ 
(Theorem \ref{thm:double}). 
To do this, we have to take into account the distinction between 
the empty simplicial complex and the void. 

Let $X$ be a metric space, $a, b\in X$, and $\ell>0$. 
Let 
\begin{equation}
\calK^\ell(X; a, b):=\Cau_{\ess}^\ell(X; a, b)\smallsetminus\{ (a, 0), (b, \ell) \}
\end{equation}
be the subposet of $\Cau_{\ess}^\ell(X; a, b)$ consisting of points $(x, t)$ strictly 
between $(a, 0)$ and $(b, \ell)$. Let us denote the order complex $\Delta\calK^\ell(X; a, b)$ 
by $K_\ell(X; a, b)$, namely, 
\begin{equation}
K_\ell(X; a, b)=\{
((x_0, t_0)<\cdots <(x_k, t_k))\mid (x_i, t_i)\in\calK^\ell(X; a, b)\}. 
\end{equation}
This set can be empty by various reasons. If the defining relations are 
inconsistent with axioms of metric space, then we suppose $K_\ell(X; a, b)$ is 
the void. More precisely, 
\begin{equation}
K_\ell(X; a, b)=\void, \mbox{ if }
d(a, b)>\ell. 
\end{equation}
Otherwise, if $d(a, b)\leq\ell$ and $K_\ell(X; a, b)$ is empty, then 
set $K_\ell(X; a, b)=\emptycpx$. 
As in Definition \ref{def:Deltadot} define the subcomplex 
$K'_\ell(X; a, b)=\Deltadot \calK^\ell(X; a, b)$ 
of $K_\ell(X; a, b)$ as 
\begin{equation}
K'_\ell(X; a, b):=
\{((x_0, t_0), \dots, (x_k, t_k))\in K_\ell(X; a, b)\mid 
d(a, x_0, x_1, \dots x_k, b)<\ell\}. 
\end{equation}
We also pose the assumption, 
\begin{equation}
K'_\ell(X; a, b)=\void, \mbox{ if }
d(a, b)\geq\ell. 
\end{equation}
Otherwise, if $d(a, b)<\ell$ and $K'_\ell(X; a, b)$ is empty, then 
set $K'_\ell(X; a, b)=\emptycpx$. 

\begin{theorem}
\label{thm:double}
Let $X$ be a metric space, $a, b\in X$, and $\ell>0$. Suppose $d(a, b)\leq\ell$. 
Then the pair 
$(|\Delta\Cau_{\ess}^\ell(X; a, b)|, |\Deltadot\Cau_{\ess}^\ell(X; a, b)|)$ is homotopy equivalent 
to the double suspension 
\[
\Sus^2\left(|K_\ell(X; a, b)|, |K'_\ell(X; a, b)|\right). 
\]
In particular, 
\begin{equation}
\MH^\ell_k(X; a, b)\cong H_{k-2}\left(|K_\ell(X; a, b)|, |K'_\ell(X; a, b)|\right), 
\end{equation}
for all $k\geq 0$. 
\end{theorem}

\begin{proof}
First, we consider the case $K_\ell(X; a, b)$ (equivalently, 
$\calK^\ell(X; a, b)$) 
is empty. If $\ell=d(a, b)$, then by assumption, $K_\ell(X; a, b)=\emptycpx$ and 
$K'_\ell(X; a, b)=\void$. Then using (\ref{eq:emptyvoid}), 
the double suspension 
is $\Sus^2(\emptycpx, \void)=([0, 1], \{0, 1\})$. 
On the other hand, since 
$\Cau_{\ess}^\ell(X; a, b)=\{(a, 0), (b, \ell)\}$, 
the pair $(|\Delta\Cau_{\ess}^\ell(X; a, b)|, 
|\Deltadot\Cau_{\ess}^\ell(X; a, b)|)$ 
is also homeomorphic to $([0, 1], \{0, 1\})$. 
If $\ell>d(a, b)$, then by assumption, 
$K_\ell(X; a, b)=K'_\ell(X; a, b)=\emptycpx$. 
Furthermore, both $\Delta\Cau_{\ess}^\ell(X; a, b)$ and 
$\Deltadot\Cau_{\ess}^\ell(X; a, b)$ 
are also empty. 

Now we suppose $K_\ell(X; a, b)$ (equivalently, 
$\calK^\ell(X; a, b)$) is nonempty. 
Then $\Cau_{\ess}^\ell(X; a, b)$ contains $(a, 0)$ and $(b, \ell)$. 
Define an intermediate poset 
$\calK^\ell(X; a, b)\subset\calP\subset\Cau_{\ess}^\ell(X; a, b)$ 
as 
\begin{equation}
\label{eq:intermP}
\calP=\calK^\ell(X; a, b)\cup\{(a, 0)\}=
\Cau_{\ess}^\ell(X; a, b)\smallsetminus\{(b, \ell) \}. 
\end{equation}
(From this point, for simplicity, we will write 
$\calK^\ell=\calK^\ell(X; a, b), K_\ell=K_\ell(X; a, b)$, etc.) 
Define the subcomplex $\Deltadot\calP$ of the order complex 
$\Delta\calP$ as 
\begin{equation}
\label{eq:intermPdot}
\Deltadot\calP=\{((x_0, t_0)<\cdots <(x_k, t_k))\in\Delta\calP\mid 
d(x_0, \dots, x_k, b)<\ell\}. 
\end{equation}
We will prove that the pair $(|\Delta\calP|, |\Deltadot\calP|)$ 
is homeomorphic to the reduced 
suspension of $(|K_\ell|, |K'_\ell|)$. 
Since $\calP$ is obtained by adding the minimum element to 
$\calK^\ell$, the order 
complex is the cone, i.e., $|\Delta\calP|=\Cone_\alpha(|K_\ell|)$, 
with the apex 
$\alpha=(a, 0)$. By definition (\ref{eq:cone}), we have 
\begin{equation}
\label{eq:DeltaP}
\Sus(|K_\ell|, |K'_\ell|)=(|\Delta\calP|, \Cone_\alpha(|K'_\ell|)\cup |K_\ell|). 
\end{equation}
It is straightforward that $\Cone_\alpha(|K'_\ell|)\cup |K_\ell|$ is 
exactly equal to $\Deltadot\calP$. Thus we have 
$\Sus(|K_\ell|, |K'_\ell|)=( |\Delta\calP|, |\Deltadot\calP| )$. 

Since $\Cau_{\ess}^\ell(X; a, b)=\calP\cup\{(b, \ell)\}$ is obtained by adding the maximum 
element to $\calP$, we can prove that the pair 
$( |\Delta\Cau_{\ess}^\ell(X; a, b)|, |\Deltadot \Cau_{\ess}^\ell(X; a, b)| )$ is 
homeomorphic to $\Sus( |\Delta\calP|, |\Deltadot\calP| )$ in a similar way. 
\end{proof}

\begin{example}
Let $G_1$ be the complete graph with the vertex set $\{a, b, c\}$ as in Figure \ref{fig:K3}. 
Let $G_2$ be the tree with the vertex set $\{a, b, c\}$ and edges $(ac)$ and $(bc)$. 
Note that $d_{G_2}(a, b)=2$. Let $\ell=2$. Then both $K_2(G_1; a, b)$ and $K_2(G_2; a, b)$ 
consist of one point $\{(c, 1)\}$ (see Example \ref{ex:Cau} for $G_1$). 

Now we consider $K'_2$. Since $d_{G_1}(a, b)<\ell$, $K'_2(G_1; a, b)=\emptycpx$. 
On the other hand, since $d_{G_2}(a, b)=\ell$, $K'_2(G_2; a, b)=\void$. Thus we have 
\[
\begin{split}
(K_2(G_1), K'_2(G_1)) &= (*, \emptycpx), \\
(K_2(G_2), K'_2(G_2)) &= (*, \void). 
\end{split}
\]
The double suspensions becomes, 
\[
\begin{split}
&\Sus^2 (*, \emptycpx)=\Sus(S^1, *)\simeq (S^2, *), \\
&\Sus^2 (*, \void)=\Sus([0, 1], \{0\})\simeq (*, *). 
\end{split}
\]
The former space has non-zero second homology (Example \ref{ex:graphs} (3)), while 
the latter space has vanishing homology group. 
\end{example}

Recall that a simplicial complex is called \emph{pure} if 
all of its maximal simplices have the same dimension. For example, triangulations of 
topological manifolds are pure. 
We also recall that a \emph{weighted graph} is a 
graph $G=(V, E)$ equipped with the edge length function 
$\rho:E\longrightarrow\R_{>0}$. Let $G$ be a connected weighted 
graph with edge lengths bounded below, that is, there exists 
$\varepsilon_0>0$ such that 
$\rho(v_1, v_2)\geq\varepsilon_0$ for any edge $(v_1, v_2)\in E$. 
Then the weighted graph determines 
a metric $d: V\times V\longrightarrow\R_{\geq 0}$ defined 
by the infimum of the lengths of paths connecting two vertices. 

By generalizing the construction in \cite[Corollary 5.12]{KY}, 
we have the following. 
\begin{corollary}
\label{cor:arbitrary}
$(i)$ 
Let $X$ be a metric space, and $a, b\in X$. Then 
\begin{equation}
\label{eq:intervalSus}
\calM^{d(a, b)}(X; a, b)\simeq\Sus^2|\Delta I(a, b)|. 
\end{equation}

$(ii)$ 
Let $Z$ be a finite dimensional simplicial complex. 
Then there exists a weighted graph $G$, vertices $a$ and $b$, 
and $\ell\geq 0$ such that 
\begin{equation}
\label{eq:arbitrarydouble}
\Sus^2 |Z|\simeq\calM^\ell (G; a, b). 
\end{equation}

$(iii)$ 
Let $Z$ be a finite dimensional pure simplicial complex. 
Then there exists a graph 
(with edge length $1$) 
$G=(V, E)$, two vertices $a, b\in V$, and $\ell\geq 0$ satisfying 
(\ref{eq:arbitrarydouble}). 
Furthermore, if $Z$ is a finite simplicial complex, we can take 
$G$ to be a finite graph. 
\end{corollary}
\begin{proof}
$(i)$ 
When $\ell=d(a, b)$, the poset  $\calK^\ell(X; a, b)$ is isomorphic to 
the interval $I(a, b)$. Hence we have the isomorphism of 
simplicial complexes $K_\ell(X; a, b)\cong \Delta I(a, b)$. 
By Theorem \ref{thm:double} and Remark \ref{rem:void}, 
we obtain (\ref{eq:intervalSus}). 

$(ii)$ 
Suppose $Z$ is $n$-dimensional. 
Let $\calF(Z)$ be the face poset of $Z$, 
that is the poset consisting of nonempty simplices. 
We extend the face poset by adding 
the minimum element $\widehat{0}$ and the maximum element 
$\widehat{1}$ to get the \emph{extended face poset} 
$\widehat{\calF}(Z)=\calF(Z)\sqcup\{\widehat{0}, \widehat{1}\}$. 
We consider the Hasse diagram of $\widehat{\calF}(Z)$ as a 
weighted graph in the following way. The length of the edge 
between a maximal simplex $\sigma$ and $\widehat{1}$ is 
\[
\rho(\sigma, \widehat{1})=n+1-\dim\sigma, 
\]
and other edges have length $1$ (Figure \ref{fig:wtdgraph}). 
Then, for $\ell=d(\widehat{0}, \widehat{1})=n+2$ and 
$(X; a, b)=(\widehat{\calF}(Z); \widehat{0}, \widehat{1})$, 
$K_\ell(X; a, b)$ 
is the barycentric subdivision of $Z$ and 
$K'_\ell=\void$. Hence 
by Theorem \ref{thm:double} and Remark \ref{rem:void}, 
we obtain (\ref{eq:arbitrarydouble}). 

\begin{figure}[htbp]
\centering
\begin{tikzpicture}


\coordinate (a1) at (0,0);
\coordinate (b1) at (1,0);
\coordinate (c1) at (2,0);
\coordinate (d1) at (1.5,0.8);

\draw (a1)--(b1)--(c1)--(d1)--(b1);

\filldraw[fill=black, draw=black] (a1) node[below] {$a$} circle (2pt);
\filldraw[fill=black, draw=black] (b1) node[below] {$b$} circle (2pt);
\filldraw[fill=black, draw=black] (c1) node[below] {$c$} circle (2pt);
\filldraw[fill=black, draw=black] (d1) node[above] {$d$} circle (2pt);

\filldraw[fill=gray!20!white, draw=black, thick] (b1)--(c1)--(d1)--cycle;

\coordinate (a) at (4,0);
\coordinate (b) at (5,0);
\coordinate (c) at (6,0);
\coordinate (d) at (7,0);

\coordinate (ab) at (4,1);
\coordinate (bc) at (5,1);
\coordinate (bd) at (6,1);
\coordinate (cd) at (7,1);

\coordinate (bcd) at (6,2);

\coordinate (P0) at (5.5,-1);
\coordinate (P1) at (5.5,3);

\filldraw[fill=black, draw=black] (a) node[left] {$a$} circle (2pt);
\filldraw[fill=black, draw=black] (b) node[left] {$b$} circle (2pt);
\filldraw[fill=black, draw=black] (c) node[right] {$c$} circle (2pt);
\filldraw[fill=black, draw=black] (d) node[right] {$d$} circle (2pt);

\filldraw[fill=black, draw=black] (ab) node[left] {$ab$} circle (2pt);
\filldraw[fill=black, draw=black] (bc) node[left] {$bc$} circle (2pt);
\filldraw[fill=black, draw=black] (bd) node[right] {$bd$} circle (2pt);
\filldraw[fill=black, draw=black] (cd) node[right] {$cd$} circle (2pt);
\filldraw[fill=black, draw=black] (bcd) node[right] {$bcd$} circle (2pt);

\filldraw[fill=black, draw=black] (P0) node[below] {$\widehat{0}$} circle (2pt);
\filldraw[fill=black, draw=black] (P1) node[above] {$\widehat{1}$} circle (2pt);

\draw (P0)--(a);
\draw (P0)--(b);
\draw (P0)--(c);
\draw (P0)--(d);
\draw (a)--(ab)--(b)--(bc)--(c)--(cd)--(d)--(bd)--(b);

\draw (ab)--node[left]{$2$} (P1);
\draw (bc)--(bcd);
\draw (bd)--(bcd);
\draw (cd)--(bcd);
\draw (P1)--(bcd);

\end{tikzpicture}
\caption{A (non-pure) simplicial complex $Z$ and 
associated weighted graph $\widehat{\calF}(Z)$.}
\label{fig:wtdgraph}
\end{figure}
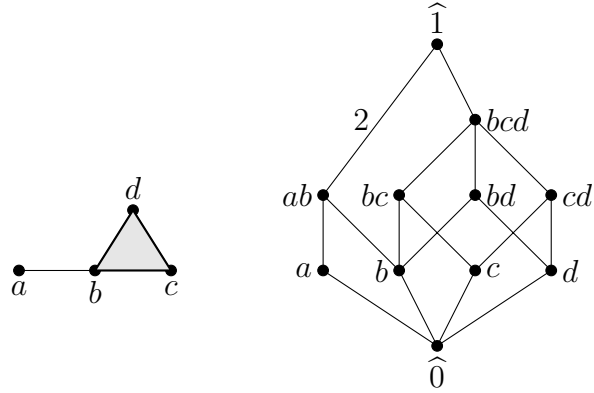

$(iii)$ 
If $Z$ is pure, the weighted graph constructed in the proof of 
$(ii)$ has edge length $1$. 
\end{proof}

As above, we can construct (the double suspension of) 
arbitrary space as the 
magnitude homotopy type with fixed terminal points 
$\calM^\ell(G; a, b)$. 
It is more difficult to control the total magnitude homotopy type 
$\calM^\ell(X)$. 
However, we can prove the following.

\begin{proposition}
\label{prop:disjoint}
Let $X$ be a metric space. Let $\ell\geq 0$. 
Suppose there exist subsets 
$Y_1, Y_2$ such that $X=Y_1\sqcup Y_2$ and 
$d(y_1, y_2)>\ell$ for any $y_1\in Y_1$ and $y_2\in Y_2$. 
Then, 
$\calM^\ell(X)=\calM^\ell(Y_1)\vee\calM^\ell(Y_2)$. 
\end{proposition}
\begin{proof}
By assumption, any sequence of length $\ell$ is 
contained in either $Y_1$ or $Y_2$. 
Thus the magnitude homotopy type decomposes. 
\end{proof}

\begin{proposition}
\label{prop:const}
Let $Z$ be a finite dimensional pure simplicial complex. 
Then there exists a metric space $X$ such that 
\[
\calM^1(X)\simeq\Sus^2|Z|\vee\Sus^2|Z|. 
\]
(Note that $\ell=1$ here.) 
Furthermore, if $Z$ is a finite simplicial complex, then we can take $X$ to be 
a finite metric space. 
\end{proposition}

\begin{proof}
Let $\widehat{\calF}(Z)$ be the Hasse diagram of the extended face poset of $Z$ as 
in the proof of Proposition \ref{cor:arbitrary}. Suppose $n=\dim Z$. Let $Z_i$ ($i=0, \dots, n$) be 
the set of $i$-dimensional simplices. Set $Z_{-1}=\{\widehat{0}\}, Z_{n+1}=\{\widehat{1}\}$. 
Let $\alpha_0, \alpha_1, \dots, \alpha_{n+1}\in\R$ 
be real numbers such that $0<\alpha_i<1, \sum_{i=0}^{n+1}\alpha_i=1$ and 
$\alpha_0, \dots, \alpha_{n+1}$ are linearly independent over $\Q$. 
By the assumption, 
$\sum_{i=0}^{n+1}r_i\alpha_i=1$ ($r_i\in\Q$) implies $r_0=r_1=\dots=r_{n+1}=1$. 
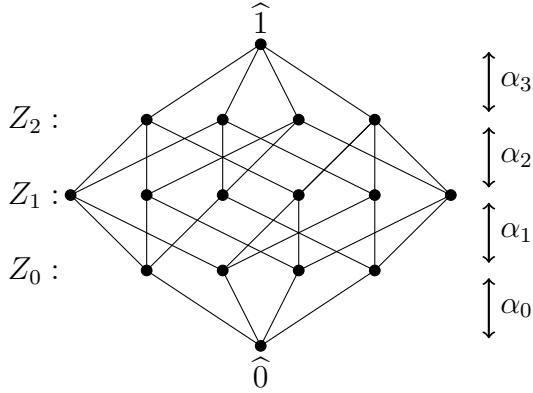
\begin{figure}[htbp]
\centering
\begin{tikzpicture}


\coordinate (a) at (0,0);
\coordinate (b) at (1,0);
\coordinate (c) at (2,0);
\coordinate (d) at (3,0);

\coordinate (ab) at (-1,1);
\coordinate (ac) at (0,1);
\coordinate (ad) at (1,1);
\coordinate (bc) at (2,1);
\coordinate (bd) at (3,1);
\coordinate (cd) at (4,1);

\coordinate (abc) at (0,2);
\coordinate (abd) at (1,2);
\coordinate (acd) at (2,2);
\coordinate (bcd) at (3,2);

\coordinate (P0) at (1.5,-1);
\coordinate (P1) at (1.5,3);

\filldraw[fill=black, draw=black] (a) circle (2pt);
\filldraw[fill=black, draw=black] (b) circle (2pt);
\filldraw[fill=black, draw=black] (c) circle (2pt);
\filldraw[fill=black, draw=black] (d) circle (2pt);

\filldraw[fill=black, draw=black] (ab) circle (2pt);
\filldraw[fill=black, draw=black] (ac) circle (2pt);
\filldraw[fill=black, draw=black] (ad) circle (2pt);
\filldraw[fill=black, draw=black] (bc) circle (2pt);
\filldraw[fill=black, draw=black] (bd) circle (2pt);
\filldraw[fill=black, draw=black] (cd) circle (2pt);

\filldraw[fill=black, draw=black] (abc) circle (2pt);
\filldraw[fill=black, draw=black] (abd) circle (2pt);
\filldraw[fill=black, draw=black] (acd) circle (2pt);
\filldraw[fill=black, draw=black] (bcd) circle (2pt);

\filldraw[fill=black, draw=black] (P0) node[below] {$\widehat{0}$} circle (2pt);
\filldraw[fill=black, draw=black] (P1) node[above] {$\widehat{1}$} circle (2pt);

\draw (P0)--(a);
\draw (P0)--(b);
\draw (P0)--(c);
\draw (P0)--(d);
\draw (P1)--(abc);
\draw (P1)--(abd);
\draw (P1)--(acd);
\draw (P1)--(bcd);

\draw (a)--(ab)--(b)--(bc)--(c)--(ac)--cycle;
\draw (a)--(ad)--(d)--(bd)--(b); 
\draw (c)--(cd)--(d); 

\draw (abd)--(ab)--(abc)--(bc)--(bcd)--(bd)--cycle;
\draw (abc)--(ac)--(acd)--(ad)--(abd); 
\draw (bc)--(bcd)--(cd)--(acd); 

\draw[<->, thick] (4.5,-0.9)-- node[right] {$\alpha_0$} (4.5,-0.1);
\draw[<->, thick] (4.5,0.1)-- node[right] {$\alpha_1$} (4.5,0.9);
\draw[<->, thick] (4.5,1.1)-- node[right] {$\alpha_2$} (4.5,1.9);
\draw[<->, thick] (4.5,2.1)-- node[right] {$\alpha_3$} (4.5,2.9);

\draw (-1, 0) node[left] {$Z_0:$ }; 
\draw (-1, 1) node[left] {$Z_1:$ }; 
\draw (-1, 2) node[left] {$Z_2:$ }; 

\end{tikzpicture}
\caption{A metric on the extended Hasse diagram $\widehat{\calF}(Z)$.}
\label{fig:meticgraph}
\end{figure}
We let the length of the edge between $Z_{i-1}$ and $Z_i$ 
be equal to $\alpha_i$, which makes 
$\widehat{\calF}(Z)$ a metric graph (see Figure \ref{fig:meticgraph}). 
Let $X=\widehat{\calF}(Z)$. 
By the assumption, the light-like paths of length $1$ 
exists only between $\widehat{0}$ and $\widehat{1}$. Thus we have 
\[
\begin{split}
\calM^1(X)
&=\calM^1(X; \widehat{0}, \widehat{1})\vee
\calM^1(X; \widehat{1}, \widehat{0})\\
&\simeq\Sus^2|Z|\vee\Sus^2|Z|. 
\end{split}
\]
\end{proof}

\begin{remark}
\label{rem:stronger}
For a fixed $\ell>0$, 
the magnitude homotopy type is strictly stronger 
than the magnitude homology group. 
Indeed, there exist metric spaces $X_1$ and $X_2$ such that 
$\MH_k^1(X_1)\cong\MH_k^1(X_2)$ for any $k\geq 0$, however, 
$\calM^1(X_1)\not\simeq\calM^1(X_2)$. 
We can construct such spaces as follows. 
Let $Z_1$ and $Z_2$ be CW complexes having 
isomorphic homology groups but not stably homotopy equivalent. 
For example, 
$Z_1=\mathbb{CP}^2$ and $Z_2=S^2\vee S^4$. 
Then $\Sus^2Z_1$ and $\Sus^2Z_2$ are 
not homotopy equivalent. 
Using Corollary \ref{cor:arbitrary} $(ii)$, we can construct 
metric spaces $X_1$ and $X_2$ having required properties. 
\end{remark}

\subsection{K\"unneth formula}
\label{sec:kunneth}

Let $X$ and $Y$ be metric spaces. We define the product metric on $X\times Y$ 
by 
\begin{equation}
d_{X\times Y}((x_1, y_1), (x_2, y_2))=d_X(x_1, x_2)+d_Y(y_1, y_2). 
\end{equation}
The following result was established by Bottinelli-Kaiser 
(and by Hepworth-Willerton for graphs), but only the total magnitude homotopy 
type version was given in their paper. So we will give the version with end points. 
\begin{theorem}
\label{thm:kunneth}
(\cite[Propositoin 8.2]{HW} for graphs, 
\cite[Proposition 4.5]{bot-kai} for metric spaces.) 
Let $X, Y$ be metric spaces and $a, a'\in X$, $b, b'\in Y$. Then 
\begin{equation}
\label{eq:kunnethwedge}
\calM^\ell(X\times Y; (a, b), (a', b'))\approx
\bigvee_{\ell_1+\ell_2=\ell}
(\calM^{\ell_1}(X; a, a')\wedge\calM^{\ell_2}(Y; b, b')), 
\end{equation}
where $\bigvee$ is the wedge sum for all 
$(\ell_1, \ell_2)$ satisfying $\ell_1\geq 0, \ell_2\geq 0$ and 
$\ell_1+\ell_2=\ell$, and $\wedge$ is the smash product of pointed spaces. 
\end{theorem}
Before the proof, recall that the product 
$P\times Q$ of two posets $P$ and $Q$ is defined by 
$(p, q)\leq (p', q')$ if and only if $p\leq_P p'$ and $q\leq_Q q'$. 
Then the order complex 
$\Delta(P\times Q)$ is a subdivision of the product of 
order complexes $\Delta P\times\Delta Q$ 
\cite[Theorem 10.21]{K}. In particular, we have a homeomorphism 
\begin{equation}
\label{eq:prod}
|\Delta(P\times Q)|\approx|\Delta P|\times|\Delta Q|. 
\end{equation}
\begin{proof}[Sketch of the proof Theorem \ref{thm:kunneth}]
First note that there is a natural map of posets (we omit the spaces 
$X, Y, X\times Y$), 
\[
\begin{split}
\Cau^{\ell_1}(a, a')\times
\Cau^{\ell_2}(b, b')&\longrightarrow
\Cau^{\ell_1+\ell_2}((a, b), (a', b'))\\ 
((x, t), (y, t'))&\longmapsto ((x, y), t+t')
\end{split}
\]
which induces a continuous map 
\begin{equation}
\bigsqcup_{\ell_1+\ell_2=\ell}
\left|
\Delta\left(
\Cau^{\ell_1}(a, a')\times
\Cau^{\ell_2}(b, b')\right)
\right|
\longrightarrow
\left|
\Delta\Cau^{\ell}((a, b), (a', b'))\right|. 
\end{equation}
By (\ref{eq:prod}), we have a continuous map 
\begin{equation}
\label{eq:decompDelta}
\bigsqcup_{\ell_1+\ell_2=\ell}
|\Delta\Cau^{\ell_1}(a, a')|
\times
|\Delta\Cau^{\ell_2}(b, b')|
\longrightarrow
\left|
\Delta\Cau^{\ell}((a, b), (a', b'))\right|. 
\end{equation}
By definition, the magnitude homotopy type of the product 
$\calM^\ell((a, b), (a', b'))$ is obtained by dividing 
the right-hand side of (\ref{eq:decompDelta}) by the subcomplex 
$\Deltadot$ consisting of sequences shorter than $\ell$. 
On the other hand, by dividing the left-hand side of 
(\ref{eq:decompDelta}) by the subspace 
\[
\bigsqcup_{\ell_1+\ell_2=\ell}
\left(|\Deltadot\Cau^{\ell_1}|
\times
|\Delta\Cau^{\ell_2}|
\cup
|\Delta\Cau^{\ell_1}|
\times
|\Deltadot\Cau^{\ell_2}|
\right), 
\]
we obtain the 
$\bigvee_{\ell_1+\ell_2=\ell}
\left(\calM^{\ell_1}(a, a')\wedge\calM^{\ell_2}(b, b')\right)$, which 
induces the homeomorphism (\ref{eq:kunnethwedge}). 
\end{proof}

Since the reduced homology group functor $\widetilde{H}_*$ is 
additive with respect to the wedge sum \cite[Corollary 2.25]{hat}, 
applying the reduced K\"unneth formula for smash product 
\cite[page 276]{hat}, we have the following K\"unneth formula for magnitude 
homology groups. 
\begin{corollary}(\cite[Theorem 5.3]{HW} for graphs and \cite[Proposition 4.3]{bot-kai} 
for metric spaces) 
We have the following split short exact sequence. 
\[
\begin{multlined}
0\longrightarrow 
\bigoplus_{\substack{i+j=n\\ \ell_1+\ell_2=\ell}}
\MH^{\ell_1}_i(X; a, a')\otimes_{\Z}
\MH^{\ell_2}_j(Y; b, b')\\
\longrightarrow
\MH^\ell_n(X\times Y; (a, b), (a', b'))\longrightarrow 
\\
\bigoplus_{\substack{i+j=n\\ \ell_1+\ell_2=\ell}}
\Tor_{\Z}(\MH_i^{\ell_1}(X; a, a'), \MH_{j-1}^{\ell_2}(Y; b, b')) \longrightarrow 0. 
\end{multlined}
\]
\end{corollary}

\subsection{Frame decomposition}
\label{sec:frame}


\begin{definition}
Let $X$ be a metric space. 
\begin{itemize}
\item[(1)] 
Let $\bm{x}=(x_0, \dots, x_k, \dots, x_n)$ be a sequence. The point $x_k$ is said to be 
\emph{smooth} if $d(x_0, \dots, \widehat{x_k}, \dots, x_n)=d(\bm{x})$, equivalently, 
$1\leq k\leq n-1$ and $d(x_{k-1}, x_k, x_{k+1})=d(x_{k-1}, x_{k+1})$. Otherwise $x_k$ is 
said to be \emph{singular} (i.e., $d(x_{k-1}, x_k, x_{k+1})>d(x_{k-1}, x_{k+1})$). 
\item[(2)] 
A sequence $(x_0, x_1, x_2, x_3)$ in $X$ is called 
a \emph{$4$-cut} if 
\[
\begin{split}
d(x_0, x_3)&< d(x_0, x_1, x_2, x_3)\\
&=d(x_0, x_2, x_3)=d(x_0, x_1, x_3). 
\end{split}
\]
In other words, $x_1$ and $x_2$ are smooth in $\bm{x}$, however, $x_1$ is singular in $(x_0, x_1, x_3)$ and 
$x_2$ is singular in $(x_0, x_2, x_3)$. 
\item[(3)] 
\[
m_X:=\inf\{d(\bm{x})\mid \bm{x} \mbox{ is a $4$-cut}\}. 
\]
(We set $m_X=\infty$ if $X$ does not have $4$-cuts.) 
\end{itemize}
\end{definition}
\begin{example}
\begin{itemize}
\item If $X$ is a subset of the Euclidean space $\R^n$, then $m_X=\infty$. 
\item If $X$ is a tree, then $m_X=\infty$. 
\end{itemize}
\end{example}
The existence of a $4$-cut is a source of the computational complexity 
in determining the magnitude homology groups. 
Indeed, Gomi \cite{gomi} proved that the non-existence of $4$-cuts 
is equivalent to the $E^2$-degeneration of the spectral sequence converging to $\MH^\ell_*(X)$. 

If $0<\ell<m_X$, Kaneta and the second author gave a decomposition of the magnitude homology 
into framed ones. Here we present a magnitude homotopy type version of the result. 

\begin{definition}
\begin{itemize}
\item[(1)] 
Let $\bm{x} = (x_0, x_1, \dots , x_n)$ be a sequence in $X$. 
Suppose $x_{i_1}, x_{i_2}, \dots, x_{i_m}$ are the list of 
all singular points of $\bm{x}$. 
(Note that $i_1=0$ and $i_m=n$.) Define $\varphi(\bm{x})$ to be 
\[
\varphi(\bm{x})=(x_{i_1}, x_{i_2}, \dots, x_{i_m}). 
\]
We call $\varphi(\bm{x})$ the \emph{frame} of $\bm{x}$. 
\item[(2)] 
A sequence $\bm{x}$ is called a \emph{singular sequence} if all points are 
singular points. 
\end{itemize}
\end{definition}
Let $\bm{x}=(x_0, \dots, x_m)$ be a sequence. 
Denote by 
$\bm{x}\smallsetminus k:=(x_0, \dots, x_{k-1}, x_{k+1}, \dots x_m)$ 
the sequence obtained by removing $k$-th element. 
If $d(\bm{x})<m_X$, then we have the following. 
\begin{equation}
d(\bm{x}\smallsetminus k)
\begin{cases}
<d(\bm{x}), &\mbox{ if $x_k$ is singular}, \\
=d(\bm{x}) \mbox{ and }\varphi(\bm{x}
\smallsetminus k)=\varphi(\bm{x}), 
&\mbox{ if $x_k$ is smooth}. 
\end{cases}
\end{equation}
So $\bm{x}\smallsetminus k$ 
is either shorter than $\bm{x}$ or having 
the same frame with $\bm{x}$. 
This enables to pose the following definition. 
We define the space consisting of simplices which has the prescribed frame. 

\begin{definition}
Let $F=(a_0, a_1, \dots, a_m)$ be a singular sequence. 
Define $S(X; F)$ as the set of simplices whose frame is $F$, namely, 
\begin{equation}
S(X; F):=\{((x_i, t_i))_{i=1}^n\in\Delta\Cau^\ell(X; a_0, a_m)
\mid \varphi((x_0, \dots, x_n))=F\}. 
\end{equation}
If $d(F)=\ell<m_X$, then 
$S(X; F)\cup\Deltadot\Cau^\ell(X; a_0, a_m)$ is a simplicial complex. 
Then define the magnitude homotopy type with frame $F$ as 
\begin{equation}
\label{eq:framedM}
\calM^\ell(X; F):=\frac{|S(X; F)\cup\Deltadot\Cau^\ell(X; a_0, a_m)|}{|\Deltadot\Cau^\ell(X; a_0, a_m)|}. 
\end{equation}
\end{definition}
Now we can decompose the magnitude homotopy type 
in terms of the framed one, 
which can be viewed as a path integral representation of the magnitude 
homotopy type. 

\begin{theorem}
\label{thm:dec}
Let $X$ be a metric space and $a, b\in X$. Suppose $0<\ell<m_X$ and $d(a, b)\leq\ell$. 
\begin{itemize}
\item[$(1)$] 
Then, 
\begin{equation}
\label{eq:dec}
\calM^\ell(X; a, b)\approx
\bigvee_{F=(a_0, \dots, a_m)}\calM^\ell(X; F), 
\end{equation}
where $F=(a_0, \dots, a_m)$ ($m\geq 0$) runs all singular sequence of length $d(F)=\ell$ 
from point $a=a_0$ to $b=a_m$. 
\item[$(2)$] 
Let $F=(a_0, \dots, a_m)$ be a singular sequence. Then 
\[
\calM^{d(F)}(X; F)\simeq
\bigwedge_{i=1}^m\calM^{d(a_{i-1}, a_i)}(X; a_{i-1}, a_i). 
\]
\end{itemize}
\end{theorem}

Before the proof of Theorem \ref{thm:dec}, let us recall the notion of the 
\emph{join} of posets $P$ and $Q$ \cite[\S 1.1]{wac-pos}. 
The join $P*Q$ is the poset whose underlying set is $P\sqcup Q$ and the order 
relation is given by $x<y$ if either (i) $x<_Py$, (ii) $x<_Qy$, or (iii) $x\in P$ and $y\in Q$. 
Note that $P*Q$ and $Q*P$ are not isomorphic in general. 
The order complex of the join is a join of 
the order complexes. 
\begin{equation}
\Delta(P*Q)\cong \Delta(P)*\Delta(Q). 
\end{equation}

\begin{proof}[Proof of Theorem \ref{thm:dec}]
$(1)$ Since any sequence $((x_i, t_i))_{i=1}^n\in\Delta\Cau^\ell(X; a, b)$ with 
$d(x_0, \dots, x_n)=\ell$ is contained in $S(X; \varphi(x_0, \dots, x_n))$, we have 
\[
\Delta\Cau^\ell(X; a, b)=\bigcup_F S(X; F)\cup\Deltadot\Cau^\ell(X; a, b), 
\]
where $F$ runs all frames of length $\ell$ from $a$ to $b$. 
Furthermore, if two singular sequences $F_1$ and $F_2$ are different, 
then $S(X; F_1)\cap S(X; F_2)\subseteq\Deltadot\Cau^\ell(X; a, b)$. 
This yields (\ref{eq:dec}). 

(2) We give the proof for $m=2$. The general case ($m\geq 3$) is similar. 
Let $F=(a_0, a_1, a_2)$ be 
a singular sequence. As in the proof of Theorem \ref{thm:double}, 
we define intermediate (pointed) spaces as 
\begin{equation}
M_1:=\frac{|\Delta I[a_0, a_1)|}{|\Delta I(a_0, a_1)|}, \ 
M_2:=\frac{|\Delta I(a_1, a_2]|}{|\Delta I(a_1, a_2)|}. 
\end{equation}
Note that $I[a_0, a_1)$ is exactly equal to $\calP$ in the proof 
of Theorem \ref{thm:double} (see the formula (\ref{eq:intermP}) 
with $a=a_0, b=a_1$). By the proof of Theorem \ref{thm:double}, 
we have $\calM^{d(a_{i-1}, a_i)}(X; a_{i-1}, a_i)\simeq\Sus 
M_i$ ($i=1, 2$). 
We also note that $S(X; F)$ is equal to the 
set of simplices in $\Delta(I[a_0, a_1)*\{a_1\}*I(a_1, a_2])$ 
containing $\{a_0, a_1, a_2\}$. Hence, 
by (\ref{eq:framedM}), $\calM^\ell(X; F)$ is equal to 
$|\Delta(I[a_0, a_1)*\{a_1\}*I(a_1, a_2])|$ divided by 
\begin{equation}
\begin{split}
\left|\Delta\left(I[a_0, a_1)*\{a_1\}*I(a_1, a_2)\right)\right|
&\cup\left|\Delta\left(|I[a_0, a_1)*I(a_1, a_2]\right)\right|\\
&\cup\left|\Delta\left(|I(a_0, a_1)*\{a_1\}*I(a_1, a_2]\right)\right|. 
\end{split}
\end{equation}
Using the relations (\ref{eq:joinsusp}), (\ref{eq:equcone}), 
and (\ref{eq:susp}), we have 
\begin{equation}
\begin{split}
\calM^{d(F)}(X; F)&\simeq\frac{M_1*\{a_1\}*M_2}{M_1*M_2}\\
&\simeq
\frac{\Cone_{a_1}(M_1* M_2)}{M_1* M_2}\\
&\simeq
\Sus(M_1*M_2)\\
&\simeq
S^1\wedge(M_1*M_2)\\
&\simeq
S^1\wedge(\Sus (M_1\wedge M_2))\\
&\simeq
(S^1\wedge M_1)\wedge (S^1\wedge M_2)\\
&\simeq
\calM^{d(a_0, a_1)}(X; a_0, a_1)\wedge 
\calM^{d(a_1, a_2)}(X; a_1, a_2). 
\end{split}
\end{equation}
\end{proof}

\begin{remark}
\label{rem:pertexp}
Using Corollary \ref{cor:arbitrary} $(ii)$, 
Theorem \ref{thm:dec} is reformulated as 
\begin{equation}
\label{eq:perthtpy}
\begin{split}
\calM^\ell(X; a, b)
&\simeq
\bigvee_{F=(a_0, \dots, a_m)}
\left(\bigwedge_{i=1}^m
\calM^{d(a_{i-1}, a_i)}(X; a_{i-1}, a_i)\right)\\
&\simeq
\bigvee_{F=(a_0, \dots, a_m)}
\left(\bigwedge_{i=1}^m\Sus^2\Delta I(a_{i-1}, a_i)
\right), 
\end{split}
\end{equation}
where $F=(a_0, \dots, a_m), (m\geq 0)$ runs all singular sequence of length $d(F)=\ell$. 
This formula may be considered as a magnitude homotopy type version of 
the formula (\ref{eq:inverse2}). 
\end{remark}

\begin{definition}
(\cite[Definition 5.1]{KY}) 
A singular sequence $F=(a_0, \dots, a_m)$ is said to be a 
\emph{thin frame} 
if $I(a_{i-1}, a_i)=\emptyset$ for all $i=1, \dots, m$. 
\end{definition}

The next result tells that if $X$ is either 
\begin{itemize}
\item a subset of the Euclidean space, or 
\item a metric space defined by a tree 
(Example \ref{ex:graphs} (1)), 
\end{itemize}
then the magnitude homotopy type $\calM^\ell(X)$ 
is a wedge of spheres. 

\begin{corollary}
\label{cor:toset}
Let $X$ be a metric space. Assume that $m_X=\infty$ and the interval poset 
$I[a, b]$ is totally ordered for any $a, b\in X$. Then 
\[
\calM^\ell(X) \simeq 
\bigvee_{\substack{F=(a_0, \dots, a_m)\\ \mbox{\scriptsize thin frame with }d(F)=\ell}}S^m. 
\]
\end{corollary}
\begin{proof}
By Theorem \ref{thm:dec}, we have 
\begin{equation}
\calM^\ell(X) \simeq 
\bigvee_{F=(a_0, \dots, a_m)}
\left(
\bigwedge_{i=1}^m\calM^{d(a_{i-1}, a_i)}(X; a_{i-1}, a_i)
\right), 
\end{equation}
where $F=(a_0, \dots, a_m)$ runs all singular sequence with length $d(F)=\ell$. 
If $I(a_{i-1}, a_i)\neq\emptyset$ for some $i=1, \dots, m$, then 
by Example \ref{ex:easy} (2), 
$\bigwedge_{i}\calM^{d(a_{i-1}, a_i)}(X; a_{i-1}, a_i)$ 
is contractible. 
On the other hand, if $I(a_{i-1}, a_i)=\emptyset$ for any 
$i=1, \dots, m$, then 
$\bigwedge_{i}\calM^{d(a_{i-1}, a_i)}(X; a_{i-1}, a_i)
\simeq (S^1)^{\wedge m}\approx S^m$. 
This completes the proof. 
\end{proof}

\section{Gluing and magnitude homotopy type}
\label{sec:applic}

\subsection{Gluing metric spaces}
\label{sec:gluing}

Let $G, H$ and $K$ be metric spaces with isometries 
$i_G:K\hookrightarrow G$ and $i_H:K\hookrightarrow H$. 
Let $X=G\cup_K H=G\cup_{i_G, i_H}H:=G\sqcup H/(i_G(k)\sim i_H(k), k\in K)$. 
If $i_G(K)\subset G$ and $i_H(K)\subset H$ are closed subsets, then 
\begin{equation}
\label{eq:gluemet}
d_X(x, y)=
\begin{cases}
d_G(x, y), \mbox{ if } x, y\in G, \\
d_H(x, y), \mbox{ if } x, y\in H, \\
\inf\limits_{k\in K} (d_G(x, i_G(k))+d_H(i_H(k), y)), \mbox{ if }x\in G, y\in H,  
\end{cases}
\end{equation}
gives a metric on $X$ \cite[Lemma 5.24]{bri-hae}. 

Denote by $\inter(G):=G\smallsetminus i_G(K)$ and 
$\inter(H):=H\smallsetminus i_H(K)$. 
In this paper, we mainly consider the gluing by the compact set $K$. 
In this case, the minimum is attained in (\ref{eq:gluemet}), namely, for $x\in\inter(G)$ and 
$y\in\inter(H)$, we have 
\begin{equation}
d_X(x, y)=\min_{k\in K}(d_G(x, i_G(k))+d_H(i_H(k), y)). 
\end{equation}
We sometimes consider $K\subset G, K\subset H, K\subset X$ and omit $i_G$ and $i_H$. 
\begin{definition}
We say that $y\in H$ \emph{projects to} $K=i_H(K)$ if there exists $\pi(y)\in K$ 
such that 
\[
d_H(k, y)=d_H(k, \pi(y))+d_H(\pi(y), y)
\]
for every $k\in K$. 
Note that the notion ``$y\in H$ projects to $K$'' 
can be defined for any metric space $H$ and subspace $K$. 
\end{definition}

\begin{proposition}
Let $X=G\cup_K H$ be a gluing by a compact set $K$ as above. 
If $y\in H$ projects to $K$, then 
$y\in X$ projects to $G$. 
\end{proposition}
\begin{proof}
By the assumption there exists $\pi(y)\in K$ such that 
\begin{equation}
\label{eq:proj}
d_X(x, y)=d_X(x, \pi(y))+d_X(\pi(y), y), 
\end{equation}
for all $x\in K$. 
We will prove that (\ref{eq:proj}) holds for every $x\in G$. 
By the triangle inequality, the left-hand side of (\ref{eq:proj}) is less than or equal to 
the right-hand side for any $x\in G$. Therefore, 
if (\ref{eq:proj}) does not hold, there exists $x\in G$ such that 
$d_X(x, y)<d_X(x, \pi(y))+d_X(\pi(y), y)$. 
Since $K$ is 
compact, there exists $k_0\in K$ such that 
$d_X(x, y)=d_G(x, k_0)+d_H(k_0, y)$ and we have 
\[
\begin{split}
d_X(x, \pi(y))+d_X(\pi(y), y)&>d_X(x, y)\\
&=d_G(x, k_0)+d_H(k_0, , y)\\
&=d_G(x, k_0)+d_H(k_0, \pi(y))+d_H(\pi(y), y), 
\end{split}
\]
which contradicts the triangle inequality 
$d_G(x, k_0)+d_H(k_0, \pi(y))\geq d_X(x, \pi(y))$. 
\end{proof}

\subsection{Discrete Morse theory and projecting matching}
\label{sec:proj}

Metric spaces $G$, $H$, $K$ are the same as in 
the previous section. 

\begin{definition}
\cite[Definition 4.3]{roff} 
Let us define $H_*\subset H$ as 
\[
H_*:=\{y\in H\smallsetminus K\mid y\mbox{ projects to }K\}. 
\]
A point in $H_*$ is called a \emph{biased point}. 
Denote the complement of $K\cup H_*$ by 
$H_0:=H\smallsetminus (K\cup H_*)$. 
A point in $H_0$ is called a \emph{non-biased point} 
(or \emph{neutral point} in \cite{roff}). 
\end{definition}

\begin{definition}
\cite[Definition 5.2]{roff} 
Let $\bm{x}=(x_0, \dots, x_n)$ be a sequence in $X$. We will say 
$\bm{x}$ is \emph{flat} if it is contained in $G\cup H_0$ or contained in $H$. 
\end{definition}

\begin{definition}
\cite[Definition 5.7]{roff} 
A sequence $\bm{x}=(x_0, \dots, x_n)$ is said to be \emph{sticky} if 
$x_0$ is contained in $\inter(G)$ and $x_n$ is contained in $H_*$, or 
vise versa, and $x_1, \dots, x_{k-1}$ are all contained in $K$. 
\end{definition}
Recall that a subsequence 
of $\bm{x}=(x_0, \dots, x_n)$ is a sequence 
$(x_i, x_{i+1}, \dots, x_j)$ of consecutive points in $\bm{x}$, where $0\leq i\leq j\leq n$. 
Let $\bm{y}=(y_0, \dots, y_m)$ be another sequence. If $x_n=y_0$ we can concatenate 
two paths $\bm{x}*\bm{y}=(x_0, \dots, x_{n-1}, x_n, y_1, \dots, y_m)$. 
\begin{proposition}
(\cite[Proposition 5.9]{roff})
\label{prop:roff}
Let $X=G\cup_K H$ be as above. Suppose $H$ projects to $K$. 
Let $\bm{x}=(x_0, \dots, x_n)$ be a sequence in $X$. Then the following 
are equivalent. 
\begin{itemize}
\item[(a)] 
$\bm{x}$ has no sticky subsequence. 
\item[(b)] 
$\bm{x}$ can be decomposed as a concatenation of paths 
$\bm{x}=\bm{x}_0*\bm{x}_1*\cdots\bm{x}_m$ such that 
each path $\bm{x}_i$ is flat and each point of concatenation is a 
non-biased point, that is a point in $H_0$. 
\end{itemize}
We call such a sequence \emph{twistable sequence}. 
\end{proposition}
\begin{proof}[Proof of Proposition \ref{prop:roff}] 
Suppose $\bm{x}=(x_0, \dots, x_n)$ has no sticky subsequence. 
If $\bm{x}$ is flat, then there is nothing to prove. 
Suppose $\bm{x}$ is not flat. Let $1\leq k\leq n$ be the maximum 
index such that $(x_0, \dots, x_k)$ is flat. 
Then $(x_0, \dots, x_k, x_{k+1})$ is not flat. By definition, 
the subsequence $(x_0, \dots, x_k)$ 
is either 
\begin{itemize}
\item[(i)] 
contained in $H$, or 
\item[(ii)] 
contained in $G\cup H_0$. 
\end{itemize}
In the case (i), 
$x_{k+1}$ is not contained in $H$, hence $x_{k+1}\in\inter(G)$. 
Let $j\leq k$ be the index such that $x_j\in\inter(H)$ and 
$x_{j+1}, \dots, x_k\in K$. Since $(x_j, x_{j+1}, \dots, x_k, x_{k+1})$ 
can not be a sticky sequence, $x_j\in H_0$. Then decompose 
$\bm{x}$ as $(x_0, \dots, x_j)*(x_j, \dots, x_n)$. The former 
part $\bm{x}_0=(x_0, \dots, x_j)$ is flat, and we continue the 
decomposition for the latter part $(x_j, \dots, x_n)$. 
In the case (ii), we have $x_{k+1}\in H_*$. 
Let $j\leq k$ be the index such that $x_j\notin K$ and 
$x_{j+1}, \dots, x_k\in K$. Then $x_j$ is contained in 
$H_0\cup\inter(G)$. 
Since $(x_j, x_{j+1}, \dots, x_k, x_{k+1})$ 
can not be a sticky sequence, we have $x_j\in H_0$. 
Then decompose 
$\bm{x}$ as $(x_0, \dots, x_j)*(x_j, \dots, x_n)$. The former 
part $\bm{x}_0=(x_0, \dots, x_j)$ is flat, and we continue the 
decomposition for the latter part $(x_j, \dots, x_n)$. 
Thus (a) implies (b). 

The converse (b)$\Longrightarrow$(a) is straightforward. 
\end{proof}

\begin{remark}
The terminology ``twistable'' is in accordance with 
the previous paper \cite{roff} 
in which the term is used in the context of sycamore twist. 
We will see later that 
it is also closely related to the critical simplices of 
the following acyclic matching 
(Theorem \ref{thm:sycamore}). 
\end{remark}

Now we construct an acyclic matching on the order complex 
$\Delta\Cau^\ell(X; a, b)$. 
Let $\bm{x}=((x_0, t_0), \dots, (x_n, t_n))$ be a light-like sequence. 
As we noticed in \S \ref{sec:const}, we will omit 
the time parameters $t_k$. 

Suppose $\bm{x}$ is not twistable. 
Then by Proposition \ref{prop:roff}, it contains 
sticky subsequences. Let $(x_i, x_{i+1}, \dots, x_j)$ be the first sticky 
subsequence of $\bm{x}$. Then either $x_i$ or $x_j$ is contained in $H_*$. 
We can construct a matching by the following rule: 
add the image of 
the point in $H_*$ by 
the projection $\pi$ to the sequence if it is not there already, 
otherwise remove it. 
More precisely, we define as follows. 
\begin{definition}
\label{def:projmatch}
Let $\bm{x}=(x_0, \dots, x_n)$ be a light-like sequence with the first 
sticky subsequence $(x_i, x_{i+1}, \dots, x_j)$. 
We define the \emph{projecting matching} as follows. 
\begin{itemize}
\item[(i)] 
If $x_i\in H_*$ and $x_{i+1}\neq\pi(x_i)$, then insert $\pi(x_i)$ to 
obtain $\bm{x'}=(x_0,\dots, x_i, \pi(x_i), x_{i+1}, \dots, x_n)$. 
Set $\bm{x}\vdash\bm{x'}$. 
\item[(ii)] 
If $x_i\in H_*$ and $x_{i+1}=\pi(x_i)$, then delete $\pi(x_i)$ to 
obtain $\bm{x''}=(x_0,\dots, x_i, x_{i+2}, \dots, x_n)$. 
Set $\bm{x''}\vdash\bm{x}$. 
\item[(iii)] 
If $x_j\in H_*$ and $x_{j-1}\neq\pi(x_j)$, then insert $\pi(x_j)$ to 
obtain $\bm{x'}=(x_0,\dots, x_{j-1}, \pi(x_j), x_j, \dots, x_n)$. 
Set $\bm{x}\vdash\bm{x'}$. 
\item[(iv)] 
If $x_j\in H_*$ and $x_{j-1}=\pi(x_j)$, then delete $\pi(x_j)$ to 
obtain $\bm{x''}=(x_0,\dots, x_{j-2}, x_j, \dots, x_n)$. 
Set $\bm{x''}\vdash\bm{x}$. 
\end{itemize}
Other than light-like sequence, we do not construct any matchings. 

The sticky sequence $(x_i, \dots, x_j)$ is called \emph{fillable} if 
$x_{i+1}\neq \pi(x_i)$ ($x_i\in H_*$) or 
$x_{j-1}\neq \pi(x_j)$ ($x_j\in H_*$). 
Otherwise, we call the sequence \emph{removable}. 
\end{definition}

\begin{proposition}
\label{prop:acyclic}
The projecting matching is a 
bounded
acyclic matching. 
\end{proposition}
\begin{proof}
Note that any sequence 
$\bm{x}=(x_0, \dots, x_n)$ can be decomposed as 
\begin{equation}
\label{eq:fsdec}
\bm{x}=\bm{w}_1*\bm{w}_2*\dots *\bm{w}_m, 
\end{equation}
such that each $\bm{w}_i$ is either flat or sticky. 
We can construct the decomposition (\ref{eq:fsdec}) as 
follows. First pick up all the sticky subsequences from $\bm{x}$. 
Then $\bm{x}$ is decomposed as a concatenation of 
sticky sequences and sequences which do not contain sticky 
subsequences. Using Proposition \ref{prop:roff}, we can 
further decompose non-sticky parts into flat sequences, and 
assume that each point of concatenation of two flat sequences is 
contained in $H_0$. 
Note also that, any point in $K$ is not the 
point of concatenation in (\ref{eq:fsdec}) because the end points 
of a sticky sequence are contained in $H_*\cup\inter(G)$. 

We first prove acyclicity. 
Suppose there exists a cycle 
\begin{equation}
\bm{x}_1\vdash\bm{y}_1\supset
\bm{x}_2\vdash\bm{y}_2\supset\dots\supset
\bm{x}_p\vdash\bm{y}_p\supset
\bm{x}_{p+1}=\bm{x}_1, 
\end{equation}
$p\geq 2$ and 
$\bm{x}_i\neq \bm{x}_j$ for $1\leq i<j\leq p$. 
Let $\bm{x}_1=(x_0, \dots, x_n).$ 
Denote the number of points of $\inter(G)$ (resp. $\inter(H)$) 
in the sequence $\bm{x}$ by $|\bm{x}|_G$ (resp. $|\bm{x}|_H$). 
Then 
$|\bm{x}_1|_G= |\bm{y}_1|_G\geq |\bm{x}_2|_G= |\bm{y}_2|_G\geq \dots
= |\bm{y}_p|_G\geq |\bm{x}_{p+1}|_G=|\bm{x}_1|_G$. 
So, these numbers are equal. Similarly, 
$|\bm{x}_1|_H= |\bm{y}_1|_H=|\bm{x}_2|_H= |\bm{y}_2|_H= \dots
= |\bm{y}_p|_H= |\bm{x}_{p+1}|_H=|\bm{x}_1|_H$. 
Hence for each 
$i\geq 1$, $\bm{x}_{i+1}$ is obtained from 
$\bm{y}_i=(y_{i, 0}, \dots, y_{i, n+1})$ by removing a point in $K$, say 
$y_{i,\alpha_i}\in K$. 
Furthermore, the removal of 
$y_{i,\alpha_i}$ from $\bm{y}_i$ 
produces a new first sticky subsequence. 
Let us consider the decomposition (\ref{eq:fsdec}) of $\bm{x}_1$. 
Suppose $\bm{w}_i=(x_q, x_{q+1}, \dots, x_r)$ 
is the first sticky subsequence of 
$\bm{x}_1$. We assume $x_q\in H_*$ 
and $x_r\in\inter(G)$. Note that $(x_q, \dots, x_r)$ is 
fillable. 
By definition, $\bm{y}_1$ is obtained by inserting $\pi(x_q)$ between 
$x_q$ and $x_{q+1}$. 
Then $\bm{w}_i':=(x_q, \pi(x_q), x_{q+1}, \dots , x_r)$ is the first 
sticky subsequence of $\bm{y}_1$, 
which is removable. 
By the construction, $\bm{w}_1, \dots, \bm{w}_{i-1}$ are 
flat. Any removal of a point of $K$ from a flat sequence 
does not produce new sticky sequence. 
Thus $\alpha_1\geq q+1$. 
Furthermore, if $\alpha_1>r+1$, then 
$\bm{x}_2$ has the 
first sticky part $\bm{w}_i'$ which is removable, and 
hence $\bm{x}_2\vdash\bm{y}_2$ is not possible. 
Therefore, $\alpha_1$ belongs to $\{{q+1}, \dots, r\}$. 
In order for the next matching $\bm{x}_2\vdash\bm{y}_2$, 
the only possibility is $\alpha_1=q+1$ and 
$x_{\alpha_1}=\pi(x_q)$, and $\bm{x}_2=\bm{x}_1$. 
This contradicts $p\geq 2$. 

Next we prove that the projecting matching is bounded. 
Suppose $p>0$ and there is a sequence 
$\bm{x}_1\vdash\bm{y}_1\supset\bm{x}_2\vdash
\bm{y}_2\supset\dots\supset
\bm{x}_p\vdash\bm{y}_p$. 
Then let us prove $p\leq |\bm{x}_1|_G+|\bm{x}_1|_H+1$. 

We again consider the decomposition (\ref{eq:fsdec}) of $\bm{x}_1$ 
and let $\bm{w}_i=(x_q, x_{q+1}, \dots, x_p)$ be the first 
sticky sequence which is fillable. 
So, $\bm{y}_1$ is of the form 
$\bm{y}_1=\bm{w}_1* \dots *\bm{w'}_i*\dots *\bm{w}_m$, 
where $\bm{w'}_i$ is obtained from $\bm{w}_i$ by inserting 
a projected point. 
In the next step, we obtain $\bm{x}_2$ by removing a point 
$y_{\alpha_1}$ from $\bm{y}_1=(y_0, \dots, y_{n+1})$. 
We will prove $y_{\alpha_1}\in\inter(G)\cup\inter(H)$. 
If not, $y_{\alpha_1}\in K$. In this case, by a 
similar argument used in the proof of acyclicity, it must be 
$\alpha_1=q+1$, and $\bm{x}_2=\bm{x}_1$. Therefore, 
the point of removal is $y_{\alpha_1}\in\inter(G)\cup\inter(H)$. 
Obviously, this procedure stops at most $|\bm{x}_1|_G+|\bm{x}_1|_H$ 
steps. 
\end{proof}

By definition, the matching is defined between light-like sequences 
which contains sticky subsequence. 
Let us denote the set of light-like 
twistable sequences from $a$ to $b$ by 
\begin{equation}
T^\ell (X; a, b):=\{\bm{x}\mid \bm{x}\mbox{ is a light-like, 
twistable sequence from $a$ to $b$ with $d(\bm{x})=\ell$}\}. 
\end{equation}
We also denote the union for all $a, b\in X$ by 
$T^\ell(X):=\bigsqcup_{a, b\in X}T^\ell(X; a, b)$. 
By Proposition \ref{prop:roff}, we 
obtain the following. 
\begin{proposition}
\label{prop:cr}
The set of critical simplices with respect to the projecting matching 
is equal to 
\[
\Deltadot\Cau^\ell(X)\cup T^\ell (X). 
\]
\end{proposition}

\subsection{Additivity (Mayer-Vietoris formula)}
\label{sec:mv}

Throughout this section (\S \ref{sec:mv}), it is assumed that 
\[
H_0=\emptyset, 
\]
equivalently, $H_*=H\smallsetminus K$, hence every element in  $\inter(H)$ 
is biased. This assumption is equivalent to 
``$H$ projects to $K$'' (\cite{HW}), which is also equivalent to 
``gated decomposition'' (\cite{bot-kai}). 
Also, as in the previous sections, 
denote the union by $X=G\cup_KH$.

\begin{lemma}
\label{lem:shortcut}
Let $a, c\in K$ and $b\in H_*$. Then $d(a, c)<d(a, b)+d(b, c)$. 
\end{lemma}
\begin{proof}
Since $b\notin K, \pi(b)\in K$, and $d(b, \pi(b))>0$, we have 
\[
\begin{split}
d(a, b)+d(b, c)
&=d(a, \pi(b))+2d(b, \pi(b))+d(\pi(b), c)\\
&>d(a, \pi(b))+d(\pi(b), c)\\
&\geq d(a, c). 
\end{split}
\]
\end{proof}

Let us define the set of simplices $S'(H)\subset\Delta\Cau^\ell(H)$ as 
\begin{equation}
S'(H):=\{((x_i, t_i))_{i=0}^n\in \Delta\Cau^\ell(H)\mid
x_k\in\inter(H)  \mbox{ for some } 0\leq k\leq n\}. 
\end{equation} 
Then obviously, we have 
\begin{equation}
\label{eq:Sprime}
\Delta\Cau^\ell(H)=S'(H)\sqcup\Delta\Cau^\ell(K). 
\end{equation} 
Note that $S'(H)$ is not necessarily a subcomplex of $\Delta\Cau^\ell(H)$. 
However, we have the following. 
\begin{proposition}
$S'(H)\cup\Deltadot\Cau^\ell(H)$ is a subcomplex of $\Delta\Cau^\ell(H)$. 
\end{proposition}
\begin{proof}
Let 
\[
\bm{x}=((x_0, t_0), \dots, (x_n, t_n))\in S'(H)
\subset\Delta\Cau^\ell(H). 
\]
If $d(x_0, \dots, x_n)<\ell$, then $\bm{x}$ and its subsets (faces) are 
clearly contained in $\Deltadot\Cau^\ell(H)$. Suppose $d(x_0, \dots, x_n)=\ell$. 
Then by the assumption, $x_k\in\inter(H)$ for some $0\leq k\leq n$. 
If  we remove $(x_k, t_k)$, then by Lemma \ref{lem:shortcut}, the resulting 
sequence is contained in $\Deltadot\Cau^\ell(H)$. 
Otherwise, it is contained in $S'(H)$. 
\end{proof}
Now we define 
\begin{equation}
\mathring{\calM}^\ell(H):=
\frac{|S'(H)\cup\Deltadot\Cau^\ell(H)|}{|\Deltadot\Cau^\ell(H)|}. 
\end{equation}
Since $S'(H)$ and $\Delta\Cau^\ell(K)$ do not have common simplices, 
we obtain the following. 
\begin{proposition}
\label{prop:decompH}
We have 
\begin{equation}
\calM^\ell(H)\approx\mathring{\calM}^\ell(H)\vee\calM^\ell(K). 
\end{equation}
\end{proposition}
\begin{proof}
\begin{equation}
\begin{split}
\calM^\ell(H)&=\frac{|\Delta\Cau^\ell(H)|}{|\Deltadot\Cau^\ell(H)|}\\
&=\frac{|S'(H)\sqcup\Delta\Cau^\ell(K)|}{|\Deltadot\Cau^\ell(H)|}\\
&\approx\frac{|(S'(H)\cup \Deltadot\Cau^\ell(H))\cup (\Delta\Cau^\ell(K)\cup \Deltadot\Cau^\ell(H))|}{|\Deltadot\Cau^\ell(H)|}\\
&\approx\frac{|S'(H)\cup \Deltadot\Cau^\ell(H)|\cup |\Delta\Cau^\ell(K)\cup \Deltadot\Cau^\ell(H)|}{|\Deltadot\Cau^\ell(H)|}\\
&\approx\frac{|S'(H)\cup \Deltadot\Cau^\ell(H)|}{|\Deltadot\Cau^\ell(H)|}\vee\frac{|\Delta\Cau^\ell(K)\cup \Deltadot\Cau^\ell(H)|}{|\Deltadot\Cau^\ell(H)|}\\
&\approx\mathring{\calM}^\ell(H)\vee\calM^\ell(K). 
\end{split}
\end{equation}
\end{proof}

\begin{theorem}
\label{thm:union}
We have the following homotopy equivalence 
\begin{equation}
\label{eq:uniondec}
\calM^\ell(X)\simeq
\mathring{\calM}^\ell(H)\vee\calM^\ell(G). 
\end{equation}
\end{theorem}
\begin{proof}
Let us define the set of simplices $S'(G, H)\subset\Delta\Cau^\ell(G\cup H)$ as 
\begin{equation}
S'(G, H):=\left\{((x_i, t_i))_{i=0}^n\in \Delta\Cau^\ell(X)\left|
\begin{array}{l}
d(x_0, \dots, x_n)=\ell, \\
x_k\in\inter(G),  x_{k'}\in\inter(H)\\ 
\mbox{for some } 0\leq k, k'\leq n
\end{array}
\right.\right\}. 
\end{equation}
Since $\inter(H)=H_*$, a sequence $(x_0, \dots, x_n)$ contains a sticky 
subsequence if and only if $\exists k, k'$ such that $x_k\in\inter(G)$ and 
$x_{k'}\in\inter(H)$. Therefore, 
$S'(G, H)$ is the set of light-like sequences that contains sticky subsequences. 
We have the following decomposition 
\begin{equation}
\label{eq:Xdecomposition}
\Delta\Cau^\ell(X)=
\left(
\Deltadot\Cau^\ell(X)\cup
\Delta\Cau^\ell(G)\cup
\Delta\Cau^\ell(H)
\right)\sqcup
S'(G, H). 
\end{equation}
Note that $\Deltadot\Cau^\ell(X)\cup
\Delta\Cau^\ell(G)\cup
\Delta\Cau^\ell(H)$ is a subcomplex of $\Delta\Cau^\ell(X)$. 
The projecting matching (Definition \ref{def:projmatch}) gives an 
acyclic matching on $\Delta\Cau^\ell(X)$ 
(Proposition\ref{prop:acyclic}). 
More precisely, the projecting matching is the set of pairs of 
simplices in $S'(G, H)$. Furthermore, $S'(G, H)$ does not  
contain critical cells. 
Therefore, by the decomposition (\ref{eq:Xdecomposition}) and 
Proposition \ref{prop:infinite}, 
$|\Deltadot\Cau^\ell(X)\cup\Delta\Cau^\ell(G)
\cup\Delta\Cau^\ell(H)|$ 
is a deformation retract of $|\Delta\Cau^\ell(X)|$, 
which induces a homotopy equivalence 
\begin{equation}
|\calM^\ell(X)|\simeq
\frac{|\Deltadot\Cau^\ell(X)\cup\Delta\Cau^\ell(G)
\cup\Delta\Cau^\ell(H)|}{|\Deltadot\Cau^\ell(X)|}. 
\end{equation}
Using the decomposition (\ref{eq:Sprime}) and 
$\Delta\Cau^\ell(K)\subset\Delta\Cau^\ell(G)$, 
we have 
\begin{equation}
|\calM^\ell(X)|\simeq
\frac{\left|\Deltadot\Cau^\ell(X)\cup\left(\Delta\Cau^\ell(G)\sqcup S'(H)\right)\right|}{|\Deltadot\Cau^\ell(X)|}. 
\end{equation}
By an argument similar to the proof of 
Proposition \ref{prop:decompH}, we have 
the homotopy equivalence (\ref{eq:uniondec})
\end{proof}
This theorem states that the magnitude homotopy type of the union is 
depending only on $\mathring{\calM}^\ell(H)$ and $\calM^\ell(G)$, 
as long as $H$ projects $K$. 
\begin{corollary}
Let $G, H, i_H:K\hookrightarrow H$ be as above 
(we assume $H_0=\emptyset$). 
Let $i_K, i'_K: K\hookrightarrow G$ be 
two isometric embeddings of $K$ into $G$. 
We construct $X$ and $X'$ by gluing $G$ and $H$ 
using $(i_K, i_H)$ and $(i'_K, i_H)$, respectively. 
Then $\calM^\ell(X)\simeq\calM^\ell(X')$. 
(Figure \ref{fig:maghomeq}) 
\end{corollary}

\begin{figure}[htbp]
\centering
\begin{tikzpicture}


\filldraw[fill=black, draw=black] (0,1) circle (3pt) ;
\filldraw[fill=black, draw=black] (1,1) circle (3pt) ;
\filldraw[fill=black, draw=black] (2,1) circle (3pt) ;
\filldraw[fill=black, draw=black] (3,1) circle (3pt) ;
\filldraw[fill=black, draw=black] (1,0) circle (3pt) ;
\filldraw[fill=black, draw=black] (2,0) circle (3pt) ;
\draw[thick] (0,1)--(1,0)--(1,1)--(1,0)--(2,0)--(3,1)--cycle;

\filldraw[fill=black, draw=black] (0,2) circle (3pt) ;
\filldraw[fill=black, draw=black] (1,2) circle (3pt) ;
\filldraw[fill=black, draw=black] (0,3) circle (3pt) ;
\filldraw[fill=black, draw=black] (1,3) circle (3pt) ;
\draw[thick] (0,1)--(0,3)--(1,3)--(1,1);

\filldraw[fill=black, draw=black] (4,1) circle (3pt) ;
\filldraw[fill=black, draw=black] (5,1) circle (3pt) ;
\filldraw[fill=black, draw=black] (6,1) circle (3pt) ;
\filldraw[fill=black, draw=black] (7,1) circle (3pt) ;
\filldraw[fill=black, draw=black] (5,0) circle (3pt) ;
\filldraw[fill=black, draw=black] (6,0) circle (3pt) ;
\draw[thick] (4,1)--(5,0)--(5,1)--(5,0)--(6,0)--(7,1)--cycle;

\filldraw[fill=black, draw=black] (5,2) circle (3pt) ;
\filldraw[fill=black, draw=black] (6,2) circle (3pt) ;
\filldraw[fill=black, draw=black] (5,3) circle (3pt) ;
\filldraw[fill=black, draw=black] (6,3) circle (3pt) ;
\draw[thick] (5,1)--(5,3)--(6,3)--(6,1);

\filldraw[fill=black, draw=black] (8,1) circle (3pt) ;
\filldraw[fill=black, draw=black] (9,1) circle (3pt) ;
\filldraw[fill=black, draw=black] (10,1) circle (3pt) ;
\filldraw[fill=black, draw=black] (11,1) circle (3pt) ;
\filldraw[fill=black, draw=black] (9,0) circle (3pt) ;
\filldraw[fill=black, draw=black] (10,0) circle (3pt) ;
\draw[thick] (8,1)--(9,0)--(9,1)--(9,0)--(10,0)--(11,1)--cycle;

\filldraw[fill=black, draw=black] (10,2) circle (3pt) ;
\filldraw[fill=black, draw=black] (11,2) circle (3pt) ;
\filldraw[fill=black, draw=black] (10,3) circle (3pt) ;
\filldraw[fill=black, draw=black] (11,3) circle (3pt) ;
\draw[thick] (10,1)--(10,3)--(11,3)--(11,1);

\end{tikzpicture}
\caption{Magnitude homotopy equivalent graphs}
\label{fig:maghomeq}
\end{figure}
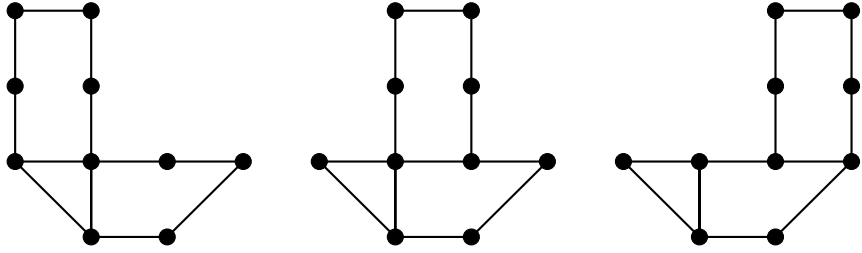

From Proposition \ref{prop:decompH} and Theorem \ref{thm:union}, 
we have 
the following Mayer-Vietoris type result (\cite{bot-kai, HW}). 
\begin{corollary}
\label{cor:mv}
Under the same assumption as Theorem \ref{thm:union}, we have the following. 
\begin{itemize}
\item[(1)] 
$\calM^\ell(X)\vee\calM^\ell(K)\simeq\calM^\ell(G)\vee\calM^\ell(H)$. 
\item[(2)] 
The inclusions $K\hookrightarrow G, H$ and 
$G, H\hookrightarrow X$ induces 
a split short exact sequence 
\[
0\longrightarrow\MH^\ell(K)
\longrightarrow
\MH^\ell(G)\oplus\MH^\ell(H)
\longrightarrow
\MH^\ell(X)
\longrightarrow
0. 
\]
\end{itemize}
\end{corollary}

\subsection{Invariance under sycamore twists}
\label{sec:inv}

In this section, we will compare the spaces obtained by gluing $G$ and $H$ 
in different ways. We return to the setup of 
\S \ref{sec:gluing} and \S \ref{sec:proj}. 
Let $H\stackrel{i_H}{\hookleftarrow}K\stackrel{i_G}{\hookrightarrow}G$ 
be isometric embeddings. Let $H_*\subset H\smallsetminus i_H(K)$ 
be the set of all biased points and 
$H_0:=H\smallsetminus (i_H(K)\sqcup H_*)$. 

\begin{definition}
Let $\alpha : K \longrightarrow K$ be any isometry. Assume that 
\begin{equation}
\label{eq:neutral}
d(h, i_H(k))=d(h, i_H(\alpha(k))), 
\end{equation}
for every $h\in H_0$ and $k\in K$. 
Construct a metric space $X$ by taking the disjoint union 
$G\sqcup H$ and identifying 
$i_G(k)$ and $i_H(k)$ ($k\in K$). Construct another metric space 
$Y$ by 
identifying $i_G(k)$ and $i_H(\alpha(k))$ ($k\in K$). We say that $X$ and 
$Y$ differ by a \emph{sycamore twist} (\cite[Definition 3.3]{roff}). 
\end{definition}

\begin{example}
Let $G$ and $H$ be the metric spaces defined by 
graphs $(V_G, E_G)$ and 
$(V_H, E_H)$. Let $\{g_+, g_-\}\in E_G$ and $\{h_+,h_-\}\in E_H$ 
be edges. 
Form a new graph $X$ by identifying $g_\pm$ with $h_\pm$. 
Similarly, form 
a new graph $Y$ by identifying $g_\pm$ with $h_\mp$. 
Then $X$ and $Y$ 
are said to differ by a Whitney twist, which is a 
special case of a sycamore twist. 
In \cite{L}, it is proved that $\Mag(X)=\Mag(Y)$ if $X$ and $Y$ 
differ by a Whitney twist. 
\end{example}

\begin{example}
Let $G$ and $H$ be graphs, 
and $p_G, q_G\in G$ and $p_H, q_H\in H$ be vertices 
as in Figure \ref{fig:sycamore}. Let $K=\{p, q\}$ be 
the metric space consisting of two points with  $d(p, q)=2$. 
Since $d_G(p_G, q_G)=d_H(p_H, q_H)=2$, we have isometric embeddings 
$G\stackrel{i_G}{\hookleftarrow}K\stackrel{i_H}{\hookrightarrow}H$. 
Let $\alpha:K\rightarrow K$ be the map $\alpha(p)=q, \alpha(q)=p$. 
Then the white vertices in Figure \ref{fig:sycamore} are the neutral points $H_0$. 
Clearly, the relation (\ref{eq:neutral}) is satisfied. Therefore, 
$X$ and $Y$ in Figure \ref{fig:sycamore} deffer by sycamore twist. 
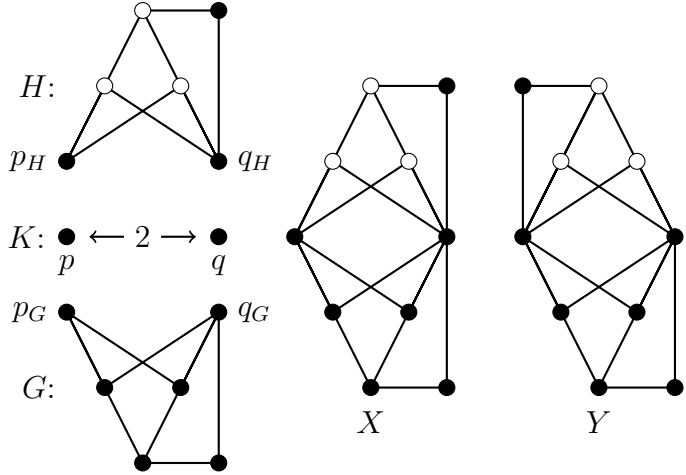
\begin{figure}[htbp]
\centering
\begin{tikzpicture}


\draw[<->, thick] (0.3, 3)-- node[fill=white] {$2$} (1.7, 3);
\filldraw[fill=black, draw=black] (0, 3) node [left] {$K$:\ \ } node[below=3pt] {$p$} circle (3pt); 
\filldraw[fill=black, draw=black] (2, 3) node[below=3pt] {$q$} circle (3pt); 
\draw (0, 1) node[left]{$G$:\ }; 
\draw (0, 5) node[left]{$H$:\ }; 

\coordinate (G1) at (0,2);
\coordinate (G2) at (2,2);
\coordinate (G3) at (0.5,1);
\coordinate (G4) at (1.5,1);
\coordinate (G5) at (1,0);
\coordinate (G6) at (2,0);

\draw[thick] (G1)--(G3)--(G2)--(G4)--(G1)--(G5)--(G6)--(G2)--(G5);

\filldraw[fill=black, draw=black] (G1) node[left=3pt] {$p_G$} circle (3pt) ;
\filldraw[fill=black, draw=black] (G2) node[right=3pt] {$q_G$} circle (3pt) ;
\filldraw[fill=black, draw=black] (G3) circle (3pt) ;
\filldraw[fill=black, draw=black] (G4) circle (3pt) ;
\filldraw[fill=black, draw=black] (G5) circle (3pt) ;
\filldraw[fill=black, draw=black] (G6) circle (3pt) ;


\coordinate (H1) at (0,4);
\coordinate (H2) at (2,4);
\coordinate (H3) at (0.5,5);
\coordinate (H4) at (1.5,5);
\coordinate (H5) at (1,6);
\coordinate (H6) at (2,6);

\draw[thick] (H1)--(H3)--(H2)--(H4)--(H1)--(H5)--(H6)--(H2)--(H5);

\filldraw[fill=black, draw=black] (H1) node[left=3pt] {$p_H$} circle (3pt) ;
\filldraw[fill=black, draw=black] (H2) node[right=3pt] {$q_H$} circle (3pt) ;
\filldraw[fill=white, draw=black] (H3) circle (3pt) ;
\filldraw[fill=white, draw=black] (H4) circle (3pt) ;
\filldraw[fill=white, draw=black] (H5) circle (3pt) ;
\filldraw[fill=black, draw=black] (H6) circle (3pt) ;

\coordinate (X1) at (3,3);
\coordinate (X2) at (5,3);
\coordinate (X3) at (3.5,2);
\coordinate (X4) at (4.5,2);
\coordinate (X5) at (4,1);
\coordinate (X6) at (5,1);
\coordinate (X7) at (3.5,4);
\coordinate (X8) at (4.5,4);
\coordinate (X9) at (4,5);
\coordinate (X10) at (5,5);

\draw (X5) node[below=5pt]{$X$}; 

\draw[thick] (X1)--(X3)--(X2)--(X4)--(X1)--(X5)--(X6)--(X2)--(X5);
\draw[thick] (X1)--(X7)--(X2)--(X8)--(X1)--(X9)--(X10)--(X2)--(X9);

\filldraw[fill=black, draw=black] (X1) circle (3pt) ;
\filldraw[fill=black, draw=black] (X2) circle (3pt) ;
\filldraw[fill=black, draw=black] (X3) circle (3pt) ;
\filldraw[fill=black, draw=black] (X4) circle (3pt) ;
\filldraw[fill=black, draw=black] (X5) circle (3pt) ;
\filldraw[fill=black, draw=black] (X6) circle (3pt) ;
\filldraw[fill=white, draw=black] (X7) circle (3pt) ;
\filldraw[fill=white, draw=black] (X8) circle (3pt) ;
\filldraw[fill=white, draw=black] (X9) circle (3pt) ;
\filldraw[fill=black, draw=black] (X10) circle (3pt) ;

\coordinate (Y1) at (6,3);
\coordinate (Y2) at (8,3);
\coordinate (Y3) at (6.5,2);
\coordinate (Y4) at (7.5,2);
\coordinate (Y5) at (7,1);
\coordinate (Y6) at (8,1);
\coordinate (Y7) at (6.5,4);
\coordinate (Y8) at (7.5,4);
\coordinate (Y9) at (7,5);
\coordinate (Y10) at (6,5);

\draw (Y5) node[below=5pt]{$Y$}; 

\draw[thick] (Y1)--(Y3)--(Y2)--(Y4)--(Y1)--(Y5)--(Y6)--(Y2)--(Y5);
\draw[thick] (Y1)--(Y7)--(Y2)--(Y8)--(Y1)--(Y9)--(Y10)--(Y1)--(Y9)--(Y2);

\filldraw[fill=black, draw=black] (Y1) circle (3pt) ;
\filldraw[fill=black, draw=black] (Y2) circle (3pt) ;
\filldraw[fill=black, draw=black] (Y3) circle (3pt) ;
\filldraw[fill=black, draw=black] (Y4) circle (3pt) ;
\filldraw[fill=black, draw=black] (Y5) circle (3pt) ;
\filldraw[fill=black, draw=black] (Y6) circle (3pt) ;
\filldraw[fill=white, draw=black] (Y7) circle (3pt) ;
\filldraw[fill=white, draw=black] (Y8) circle (3pt) ;
\filldraw[fill=white, draw=black] (Y9) circle (3pt) ;
\filldraw[fill=black, draw=black] (Y10) circle (3pt) ;

\end{tikzpicture}
\caption{An example of sycamore twist.}
\label{fig:sycamore}
\end{figure}
\end{example}

\begin{theorem}
\label{thm:sycamore}
Suppose metric spaces $X$ and $Y$ differ by a sycamore twist. 
Then, there is a bijection between critical cells of projecting matching 
of $\calM^\ell(X)$ and $\calM^\ell(Y)$ which preserves the dimensions 
of cells. 
\end{theorem}
\begin{proof}
In view of Proposition \ref{prop:cr}, it is enough to show that there exists a 
bijection 
\[
T^\ell(X)\stackrel{\cong}{\longrightarrow} T^\ell(Y) 
\]
which preserves the dimensions of cells (degrees of sequences). 
This claim can be proved in a similar manner as \cite[Proposition 5.6]{roff}. 
First let us define maps $\tau_G, \tau_H:X\longrightarrow Y$ 
as follows (see also Figure \ref{fig:tau}). 
(Note that, here, $X$ and $Y$ are defined as 
$G\sqcup H/\sim$, where $\sim$ is certain equivalence relation. 
Therefore, any point in $X$ can be expressed as $\overline{x}$ with 
$x\in G\sqcup H$. One can easily check the following is well-defined 
on $i_G(K)\sqcup i_H(K)$.) 
\begin{equation}
\tau_G(\overline{x})=
\begin{cases}
\overline{x}, \mbox{ if $x\in G\smallsetminus i_G(K)$}, \\
\overline{x}, \mbox{ if $x\in i_G(K)$}, \\
\overline{x}, \mbox{ if $x\in H\smallsetminus i_H(K)$}, \\
\overline{i_H(\alpha(i_H^{-1}(x)))}, \mbox{ if $x\in i_H(K)$}, 
\end{cases}
\tau_H(\overline{x})=
\begin{cases}
\overline{x}, \mbox{ if $x\in G\smallsetminus i_G(K)$}, \\
\overline{i_G(\alpha^{-1}(i_G^{-1}(x)))}, \mbox{ if $x\in i_G(K)$}, \\
\overline{x}, \mbox{ if $x\in H\smallsetminus i_H(K)$},  \\
\overline{x}, \mbox{ if $x\in i_H(K)$}. 
\end{cases}
\end{equation}
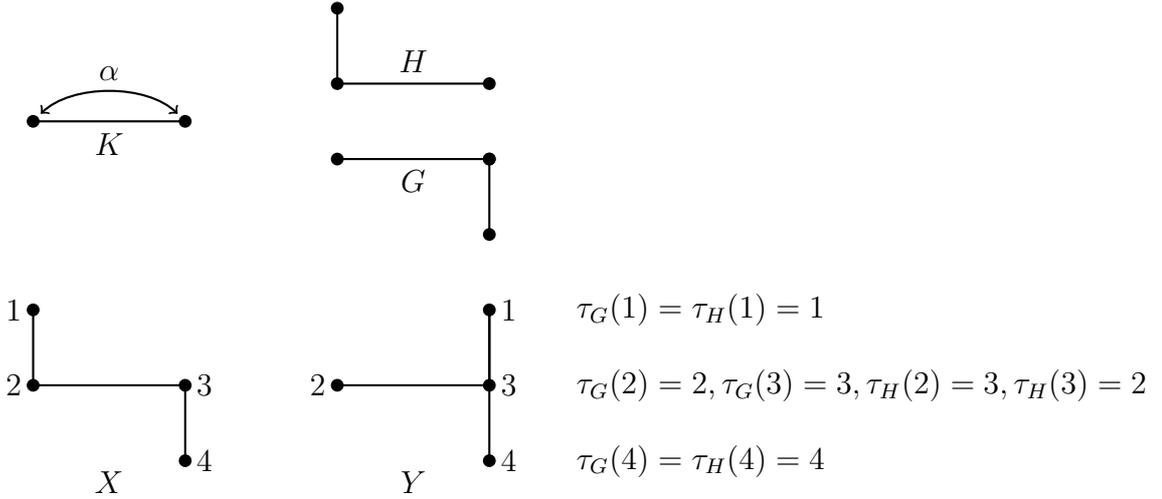
\begin{figure}[htbp]
\centering
\begin{tikzpicture}


\filldraw[thick] (0, 4.5) circle (2pt); 
\filldraw[thick] (2, 4.5) circle (2pt); 
\draw[thick] (0, 4.5)-- node[below] {$K$} (2, 4.5); 
\draw[thick, <->] (0.1,4.6) .. node[above] {$\alpha$} controls (0.5,5) and (1.5,5) .. (1.9,4.6); 


\draw[thick] (4,4)-- node[below] {$G$} (6,4)--(6,3); 
\filldraw[thick] (4, 4) circle (2pt); 
\filldraw[thick] (6, 4) circle (2pt); 
\filldraw[thick] (6, 3) circle (2pt); 

\filldraw[thick] (4, 5) circle (2pt); 
\filldraw[thick] (6,5) circle (2pt); 
\filldraw[thick] (4, 6) circle (2pt); 
\filldraw[thick] (6, 4) circle (2pt); 

\draw[thick] (4,6)--(4,5)-- node[above] {$H$} (6,5); 

\draw (1,0) node[below] {$X$}; 
\filldraw[thick] (0, 2) node[left] {$1$} circle (2pt); 
\filldraw[thick] (0, 1) node[left] {$2$} circle (2pt); 
\filldraw[thick] (2, 1) node[right] {$3$} circle (2pt); 
\filldraw[thick] (2, 0) node[right] {$4$} circle (2pt); 
\draw[thick] (0,2)--(0,1)--(2,1)--(2,0); 

\draw (5,0) node[below] {$Y$}; 
\filldraw[thick] (4, 1) node[left] {$2$} circle (2pt); 
\filldraw[thick] (6, 1) node[right] {$3$} circle (2pt); 
\filldraw[thick] (6, 2) node[right] {$1$} circle (2pt); 
\filldraw[thick] (6, 0) node[right] {$4$} circle (2pt); 
\draw[thick] (4,1)--(6,1)--(6,2)--(6,0); 


\draw (7,2) node [right] {$\tau_G(1)=\tau_H(1)=1$}; 
\draw (7,1) node [right] {$\tau_G(2)=2, \tau_G(3)=3, \tau_H(2)=3, \tau_H(3)=2$}; 
\draw (7,0) node [right] {$\tau_G(4)=\tau_H(4)=4$}; 

\end{tikzpicture}
\caption{The maps $\tau_G$ and $\tau_H$.}
\label{fig:tau}
\end{figure}

By Proposition \ref{prop:roff}, $\bm{x}\in T^\ell(X)$ can be expressed as a 
concatenation 
\begin{equation}
\bm{x}=\bm{x}_1*\bm{x}_2*\cdots *\bm{x}_m, 
\end{equation}
of flat sequences $\bm{x}_i$ 
such that each point of concatenation is contained in $H_0$. 
If $\bm{x}_i$ is contained in $G\cup H_0$, then by the assumption (\ref{eq:neutral}) 
of the sycamore twist, $\tau_G(\bm{x}_i)$ has the same length with $\bm{x}_i$. 
Define $\tau(\bm{x}_i)$ by 
\begin{equation}
\tau(\bm{x}_i):=
\begin{cases}
\tau_G(\bm{x}_i), \mbox{ if $\bm{x}\subset G\cup H_0$}, \\
\tau_H(\bm{x}_i), \mbox{ otherwise, }
\end{cases}
\end{equation}
and 
\begin{equation}
\tau(\bm{x}):=\tau(\bm{x}_1)*\cdots *\tau(\bm{x}_m). 
\end{equation}
This gives a desired bijection $T^\ell(X)\longrightarrow T^\ell(Y)$. 
\end{proof}

\begin{corollary}
Suppose $G$ and $H$ are finite metric space. Let $X$ and $Y$ differ by a sycamore twist. Then, 
\begin{itemize}
\item[$(1)$] 
$\widetilde{\chi}(\calM^\ell(X))=\widetilde{\chi}(\calM^\ell(Y))$. 
\item[$(2)$] 
The magnitudes of $X$ and $Y$ coincide 
$\Mag(X)=\Mag(Y)$. 
\end{itemize}
\end{corollary}
\begin{proof}
When $X$ is a finite metric space, $\calM^\ell(X)$ is a finite CW complex 
for any $\ell\geq 0$. 
Then the Euler characteristic 
$\widetilde{\chi}(\calM^\ell(X))$ 
is determined by the number of critical cells of each dimension. 
By Theorem \ref{thm:sycamore}, 
there is a dimension preserving bijection between the set 
critical cells of $X$ and $Y$. 
This yields $(1)$. By Proposition \ref{prop:inverse}, we also have $(2)$. 
\end{proof}

\begin{remark}
As far as the authors know, it is still an open question whether 
$\MH^\ell_k(X)$ and 
$\MH^\ell_k(Y)$ are isomorphic when $X$ and $Y$ 
differ by a sycamore twist. 
A similar question on the magnitude homotopy types 
$\calM^\ell(X)$ and $\calM^\ell(Y)$ is also open. 
\end{remark}

\medskip

\noindent
{\bf Acknowledgements.} 
Yu Tajima was supported by JST SPRING, Grant Number JPMJSP2119. 
Masahiko Yoshinaga 
was partially supported by JSPS KAKENHI 
Grant Numbers JP22K18668, JP19K21826, JP18H01115. 
Part of this work was carried out while the authors were 
staying at Okayama University for a lecture series on magnitude homology 
by the second author in November 2022. 
We would like to express our gratitude to 
Professor Masao Jinzenji for the 
invitation and to the participants of the lecture series for inspiring conversations. 
We also would like to thank 
Professors Yasuhiko Asao, Kiyonori Gomi, 
Dmitry Feichtner-Kozlov, Paul M\"ucksch, 
Shin-ichi Ohta, Adri\'an Do\~na Mateo 
for helpful discussions and useful information. 
We deeply appreciate the referee(s) for careful reading 
and lots of valuable suggestions on the paper.

\end{document}